\documentclass[fleqn,11pt,letterpaper]{article}
\usepackage{amsfonts}
\usepackage{amssymb}
\usepackage{amsmath}
\usepackage{amsthm}
\usepackage{enumitem}
\usepackage{multirow}
\usepackage{float}
\usepackage{graphicx}
\usepackage{epstopdf}
\usepackage{color}
\usepackage[margin=2.57cm]{geometry}
\definecolor{darkred}{rgb}{0.5,0.2,0.2}
\usepackage[pdftitle={Semiparametric estimation of fractal index},pdfauthor={Mikkel Bennedsen},colorlinks=TRUE,allcolors=darkred]{hyperref}
\usepackage{booktabs}
\usepackage{pdflscape}
\usepackage[round]{natbib}

\usepackage{tabularx}

\usepackage{float}
\floatstyle{plaintop}
\restylefloat{table}

\restylefloat{table}

\theoremstyle{plain}
\newtheorem{theorem}{Theorem}[section]

\newtheorem{corollary}{Corollary}[section]

\newtheorem{lemma}{Lemma}[section]
\newtheorem{proposition}{Proposition}[section]

\theoremstyle{definition}
\newtheorem{example}{Example}[section]

\theoremstyle{remark}
\newtheorem{remark}{Remark}[section]

\def\F{{\mathcal F}}

\newcommand{\N}{\mathbb{N}}

\newcommand{\R}{\mathbb{R}}

\newcommand{\E}{\mathbb{E}}
\newcommand{\BSS}{\mathcal{BSS}}

\def\bi{\begin{itemize}}
\def\ei{\end{itemize}}

\allowdisplaybreaks
\graphicspath{{Figures/}}
\numberwithin{equation}{section}

\newcounter{nameOfYourChoice}

\newif\ifi

\begin{document}

\title{Semiparametric estimation and inference on the fractal index of Gaussian and conditionally Gaussian time series data\footnote{The author wishes to thank Professor Asger Lunde and Dr. Mikko S. Pakkanen for insightful discussions relating to fractal processes. The research has been supported by CREATES (DNRF78), funded by the Danish National Research Foundation.}}

\author{Mikkel Bennedsen\thanks{
Department of Economics and Business Economics and CREATES, 
Aarhus University, 
Fuglesangs All\'e 4,
8210 Aarhus V, Denmark.
E-mail:\
\href{mailto:mbennedsen@econ.au.dk}{\nolinkurl{mbennedsen@econ.au.dk}} 
}}

\maketitle

\begin{abstract}
We study a well-known estimator of the fractal index of a stochastic process. Our framework is very general and encompasses many models of interest; we show how to extend the theory of the estimator to a large class of non-Gaussian processes. Particular focus is on clarity and ease of implementation of the estimator and the associated asymptotic results, making it easy for practitioners to apply the methods. We additionally show how measurement noise in the observations will bias the estimator, potentially resulting in the practitioner erroneously finding evidence of fractal characteristics in a time series. We propose a new estimator which is robust to such noise and construct a formal hypothesis test for the presence of noise in the observations. Finally, the methods are illustrated on two empirical data sets; one of turbulent velocity flows and one of financial prices.
\end{abstract}

\vspace*{1.5em}


\vspace*{1em}


\vspace*{1em}



\section{Introduction}
Fractal-like models are used in a wide array of applications such as in the characterization of surface smoothness/roughness \citep{CH94}, in the study of turbulence \citep{CHPP13}, and many others \citep[e.g.,][]{burrough81,mandelbrot82,falconer90}. Most recently, these models have attracted attention in mathematical finance, as models of stochastic volatility \citep[e.g.,][]{GJR14,BFG15,BLP16,BLP15,JMM17}. In such applications, it is imperative to be able to estimate and conduct inference on the key parameter in these models, the \emph{fractal index}. Many estimators of this parameter exist \citep[see][for a survey]{GSP12}; however, the underlying assumptions behind the various estimators, as well as their asymptotic properties, are often different and rarely stated in a clear and concise manner. These facts can make analysis difficult for the practitioner, as well as for the researcher.

This paper aims at making empirical analysis in applications, such as the ones mentioned above, easier. We clearly lay out a large and coherent framework -- including the valid underlying assumptions -- for analysing time series data which are potentially fractal-like. We focus on a specific estimator, which is arguably the most widely used in practice and which in our experience is the most accurate. Further, the estimator is easy to implement -- it relies on a simple OLS regression -- and its asymptotic properties are easy to apply. Our hope is that this will provide a transparent guide to analyzing fractal data using sound statistical methods.


The main contribution of the  paper is to lay out the theory of the estimator and provide the theoretical underpinnings of it, stating the results in a manner so that application of the results becomes straight forward. For this, we rely heavily on earlier theoretical work on the increments of fractal processes, most notably \cite{BNCP09} and \cite{BNCP11}. We further investigate the estimator numerically to gauge it's properties when applied to data, leading to a number of practical recommendations for implementation. Most importantly, we advocate a different choice of bandwidth parameter for the estimator, than what is generally accepted practice in the literature, cf. Section \ref{sec:band}.


In their survey of the asymptotic theory of various estimators of the fractal index, \cite{GSP12} section 3.1., report that ``a general non-Gaussian theory remains lacking". The second contribution of this paper is to extend the estimation theory beyond the Gaussian paradigm. We accomplish this by \emph{volatility modulation} which turns out to be a convenient way of extending the theory to a large class of non-Gaussian processes. As will be seen, this results in \emph{conditionally} Gaussian processes for which the fractal theory continues to hold. Again we clearly lay out the relevant assumptions and focus on the interpretation of the results and implementation of the methods.

The final contribution of the paper is an in-depth study of the case where the data are contaminated by noise, such as measurement noise. We prove that noise will bias estimates of the fractal index downwards, thereby making noise-contaminated data look more rough than the underlying process actually is. We go on to propose a novel way to construct an estimator which is robust to noise in the observations. The new estimator also relies on an OLS regression and is just as easy to implement as the standard (non-robust) estimator studied in the first part of the paper. We present the asymptotic theory concerning the robust estimator and propose a hypothesis test, which can be used to formally test for the presence of noise in the observations. 

The rest of the paper is structured as follows. Section \ref{sec:setup} presents the mathematical setup and assumptions and gives some examples of the kind of processes we have in mind. The section then goes on to consider some extensions to the basic setup, most notably the extension to non-Gaussian processes. Section \ref{sec:bootstrap} presents the semiparametric estimator of the fractal index and it's asymptotic properties. Then, in Section \ref{sec:assNoise}, we consider the case where the observations have been contaminated by noise and present asymptotic theory for a new estimator in this case; Section \ref{sec:testNoise} presents a formal test for the presence of noise. Section \ref{sec:sims} contains small simulation studies, illustrating the finite sample properties of the asymptotic results presented in the paper. Finally, Section \ref{sec:emp} contains two illustrations of the methods: the first using measurements of the longitudinal component of a turbulent velocity field, and the second using a time series of financial prices. Section \ref{sec:concl} concludes and gives some directions for future study. Proofs of technical results and some mathematical derivations are given in an appendix.

\section{Setup}\label{sec:setup}
Let $(\Omega,\F, \mathbb{P})$ be a probability space satisfying the usual assumptions and supporting $X$, a one-dimensional, zero-mean, stochastic process with stationary increments. Define the \emph{$p$'th order variogram} of $X$:
\begin{align*}
\gamma_p(h;X) := \E[|X_{t+h} - X_t|^p], \quad h \in \R.
\end{align*}
As we intend to make use of the theory developed in \cite{BNCP09,BNCP11} we adopt the assumptions of those papers. The assumptions are standard in the literature on fractal processes and are as follows.

\begin{enumerate}[label=(A\arabic*),ref=A\arabic*,leftmargin=3em]
\item \label{ass:1}
For some $\alpha \in \left(-\frac{1}{2},\frac{1}{2}\right)$,
\begin{align}\label{eq:scaling1}
\gamma_2(x;X) = x^{2\alpha+1}L(x),	\quad x \in (0,\infty),
\end{align}
where $L : (0,\infty) \to [0,\infty)$ is continuously differentiable and bounded away from zero in a neighborhood of $x=0$. The function $L$ is assumed to be slowly varying at zero, in the sense that $\lim_{x\rightarrow 0} \frac{L(tx)}{L(x)} = 1$ for all $t>0$. 
\item \label{ass:2}
	$\frac{d^2}{dx^2} \gamma_2(x;X) = x^{2\alpha-1}L_2(x)$ for some slowly varying (at zero) function $L_2$, which is continuous on $(0,\infty)$.\footnote{Following \cite{BHLP16} this assumption is replaced by the following in the case $\alpha = 0$: (\ref{ass:2}') $\frac{d^2}{dx^2} \gamma_2(x;X) = f(x) L_2(x)$, where $L_2$ is as in (\ref{ass:2}), and the function $f$ is such that $|f(x)|\leq C x^{-\beta}$ for some constants $C>0$ and $\beta>1/2$.}
\item \label{ass:3} \setcounter{nameOfYourChoice}{3}
	There exists $b\in(0,1)$ with
\begin{align*}
\limsup_{x\rightarrow 0} \sup_{y \in [x,x^b]} \left| \frac{L_2(y)}{L(x)} \right| < \infty.
\end{align*}
\item \label{ass:99}  \setcounter{nameOfYourChoice}{4}
	There exists a constant $C>0$ such that the derivative $L'$ of $L$ satisfies
\begin{align*}
|L'(x)| \leq C\left(1+x^{-\delta}\right), \quad x \in (0,1],
\end{align*}
for some $\delta \in (0,1/2)$.
\end{enumerate}

\begin{remark}
The technical assumption (\ref{ass:3}) can be replaced by the weaker assumption
\begin{align*}
\left|  \frac{\gamma_2((j+1)/n;X) -2\gamma_2(j/n;X)+ \gamma_2((j-1)/n;X)}{2\gamma_2(1/n;X)} \right| \leq r(j), \quad \frac{1}{n} \sum_{j=1}^n r(j)^2 \rightarrow 0, \quad n \rightarrow \infty,
\end{align*}
for some sequence $r(j)$.
\end{remark}

\begin{remark}
The technical assumption (\ref{ass:99}) is only needed for the asymptotic normality of the estimator of $\alpha$ and not for consistency.
\end{remark}

The parameter $\alpha \in (-1/2,1/2)$ is termed the fractal index because it, under mild assumptions, is related to the \emph{fractal dimension} $D = \frac{3}{2} - \alpha$ of the sample paths of the process $X$ \citep{falconer90,GSP12}. It is also refered to as the \emph{roughness index} of $X$, since the value of $\alpha$ reflects itself in the \emph{pathwise properties} of $X$, as the following result formalizes.
\begin{proposition}\label{lem:holder}
Let $X$ be a Gaussian process with stationary increments satisfying (\ref{ass:1}) with fractal index $\alpha \in (-1/2,1/2)$.  Then there exists a modification of $X$ which has locally H\"older continuous trajectories of order $\phi$ for all $\phi \in \left(0, \alpha+\frac{1}{2}\right)$.
\end{proposition}


Proposition \ref{lem:holder} shows that $\alpha$ controls the degree of (H\"older) continuity of $X$. In particular, negative values of $\alpha$ corresponds to $X$ having \emph{rough} paths, while positive values of $\alpha$ corresponds to  \emph{smooth} paths. It is well known that the Brownian motion has $\alpha = 0$. In Table \ref{tab:ex} we give some parametric examples of the kind of processes we have in mind and comment on how they fit into the setup of the present paper; the examples are taken from Table 1 in \cite{GSP12}.

\begin{table}
\footnotesize
\caption{Parametric examples of Gaussian fractal processes}
\begin{center}
\scriptsize
\begin{tabularx}{.99\textwidth}{@{\extracolsep{\stretch{1}}}llll@{}} 
\toprule
Class & Autocorrelation function & Slowly varying function & Parameters\\ 
\midrule
fBm & $-$ & $L(x) = \beta$ & $\alpha \in (-1/2,1/2)$ \\
Mat\'ern & $\rho(x) = \frac{2^{-\alpha+ 1/2}}{\Gamma(\alpha+1/2)}|\beta x|^{\alpha+1/2}K_{\alpha+1/2}(|\beta x|)$ & $L(x) = 2 x^{-2\alpha-1}(1-\rho(x))$ & $\alpha \in (-1/2,1/2)^a$ \\
Powered exp. & $\rho(x) =\exp \left(-|\beta x|^{2\alpha+1} \right)$ & $L(x) =2 x^{-2\alpha-1}(1-\rho(x))$ & $\alpha \in (-1/2,1/2)^b$ \\
Cauchy & $\rho(x) = \left( 1+|\beta x|^{2\alpha+1} \right)^{-\frac{\tau}{2\alpha + 1}}$ & $L(x) =2 x^{-2\alpha-1}(1-\rho(x))$ & $\alpha \in (-1/2,1/2)^b$, $\tau>0$ \\
Dagum & $\rho(x) = 1 - \left(  \frac{|\beta x|^{2\tau+1}}{1+|\beta x|^{2\tau+1}}\right)^{\frac{2\alpha +1}{2\tau+1}}$ & $L(x) =2 x^{-2\alpha-1}(1-\rho(x))$ & $\tau\in (-1/2,1/2)^c$, $\alpha \in (-1/2,\tau)$ \\
\bottomrule 
\end{tabularx}
\end{center}
{\footnotesize Parametric examples of Gaussian fractal processes. ``fBm" is the fractional Brownian motion; ``Powered exp." is the powered exponential process. $\beta>0$ is a scale parameter and $\alpha$ is the fractal index. The processes fulfill assumptions (\ref{ass:1})--(\ref{ass:3}) for the parameter ranges given in the rightmost column; a letter superscript denotes whether the parameter ranges are different under (\ref{ass:99}). $a$: (\ref{ass:99}) valid for $\alpha \in (-1/2,1/4)$. $b$: (\ref{ass:99}) valid for $\alpha \in (-1/4,1/2)$. $c$: (\ref{ass:99}) valid for $\tau \in [-1/4,1/2)$.} 
\label{tab:ex}
\end{table}

To get an intuitive understanding of how the trajectories of the fractal processes look, and in particular how the value of $\alpha$ reflects itself in the roughness of the paths, Figure \ref{fig:paths} plots three simulated trajectories of the Mat\'ern process. It is evident how negative values of $\alpha$ correspond to very rough paths, while the paths become smoother as $\alpha$ increases.

\begin{figure}[!t] 
\centering 
\includegraphics[width=\textwidth]{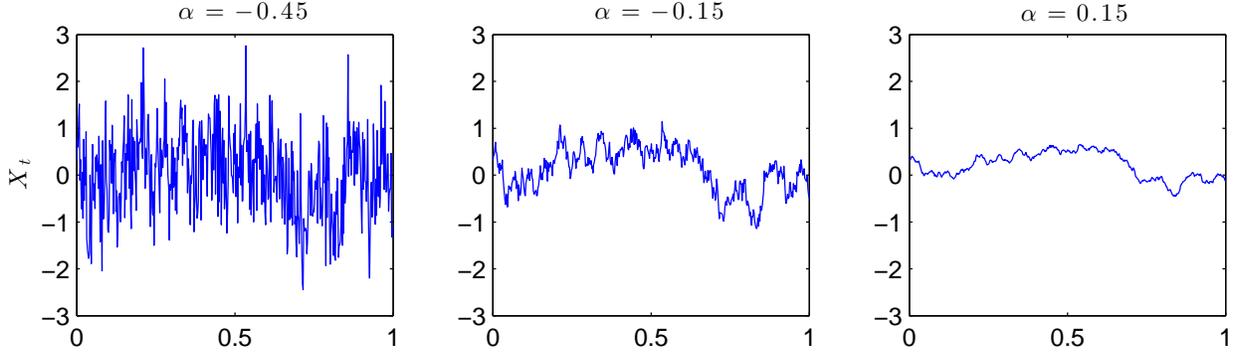}
\caption{Simulations of the unit-variance Mat\'ern process, cf. Table \ref{tab:ex}, with $\beta = 1$, $\alpha$ as indicated above the plots, and $n = 500$ observations on the unit interval. The same random numbers were used for all three instances.}
\label{fig:paths}
\end{figure}

The processes in Table \ref{tab:ex} are all Gaussian. However, in many applications it is preferable to have rough processes which are both fractal and non-Gaussian \citep[][section 3.1.]{GSP12}. In the following section we suggest an extension to the above setup that explicitly results in non-Gaussian processes with fractal properties, by considering processes which are \emph{volatility modulated}.

\subsection{Extension to stochastic volatility processes}\label{sec:SV}
A flexible way to introduce non-Gaussianity of processes for which the theory of the fractal index continues to hold, is through volatility modulation. Following \cite{BNCP09}, consider processes of the form
\begin{align}\label{eq:Xsv}
X_t = X_0 +   \int_{0}^t \sigma_s dG_s, \quad t \geq 0,
\end{align}
where $X_0 \in \R$, $\sigma = (\sigma_t)_{t\geq 0}$ is a stochastic volatility process, and $G =  (G_t)_{t\geq 0}$ is a zero-mean Gaussian process with stationary increments satisfying (\ref{ass:1})--(\ref{ass:99}), e.g. one of the processes from Table \ref{tab:ex}. The modulation of the increments of $G$ by the stochastic volatility process is a convenient way of introducing non-Gaussianity. To see this, note that the marginal distribution of $X_t$, conditional on the past of the stochastic volatility process and the starting value $X_0$, is
\begin{align*}
X_t |(\sigma_s, s \in [0,t]; X_0) \sim N\left(X_0, \int_0^{t} \sigma_{x}^2 dx \right), \quad t \geq 0.
\end{align*}
In other words, the marginal distribution of $X_t$ is a \emph{normal mean-variance mixture} distribution, where the distribution of the stochastic process $\sigma$ and initial value $X_0$ determine the mixture.

For the integral in \eqref{eq:Xsv} to be well defined (in a pathwise Riemann-Stieltjes sense), we require that $\sigma$  has finite $q$-variation for some $q < \frac{1}{1/2-\alpha}$. Intuitively, this means that the ``more rough" $G$ is, the ``less rough" $\sigma$ can be. Under these conditions on $G$ and $\sigma$, the process $X$ in \eqref{eq:Xsv} will inherit the fractal properties of the driving process $G$, as shown in \cite{BNCP09}.
 
For the central limit theorems developed below to hold, we further require another assumption on $\sigma$.

\begin{enumerate}[label=(SV),ref=SV]
\item \label{ass:4} For any $q>0$, it holds that
\begin{align*}
\E[ |\sigma_t - \sigma_s|^q] \leq C_q |t-s|^{\xi q}, \quad t,s\in \R,
\end{align*}
for some $\xi>0$ and $C_q>0$.
\end{enumerate}

As pointed out in \cite{BLP16}, the requirement that $\sigma$  has finite $q$-variation for a $q < \frac{1}{1/2-\alpha}$ can be quite restrictive. For instance, if $\alpha < 0$ (i.e., $G$ is rough) then $\sigma$ can not be driven by a standard Brownian motion. A very convenient process, which does not have these restrictions and which is very tractable, is the \emph{Brownian semistationary process}, which we consider next. 

\subsubsection{The Brownian semistationary process}\label{sec:BSS}
Consider $X$, the (volatility modulated) Brownian semistationary ($\BSS$) process  \citep{BNSc07,BNSc09}, defined as
\begin{align}\label{eq:Xbss}
X_t = \int_{-\infty}^t g(t-s) \sigma_s dW_s, \quad t\geq 0,
\end{align}
where $W$ is a Brownian motion on $\R$, $\sigma = (\sigma_t)_{t\in \R}$ a stationary process, and $g$ a Borel measurable function such that $\int_{-\infty}^t g(t-s)^2\sigma_s^2 ds < \infty$ a.s. See, e.g., \cite{BLP16} for further details of the $\BSS$ process. The $\BSS$ process is also a normal mean-variance mixture:
\begin{align*}
X_t |(\sigma_s, s\leq t) \sim N\left(0, \int_0^{\infty} g(x)^2 \sigma_{t-x}^2 dx \right), \quad t \geq 0.
\end{align*}

It is interesting to note that \cite{BNBV13} show that for a particular choice of kernel function $g$ and stochastic volatility process $\sigma$, $X$ will have a marginal distribution of the ubiquitous Normal Inverse Gaussian type.

We need to impose some technical assumptions on the kernel function $g$. They are as follows.

\begin{enumerate}[label=(BSS),ref=BSS]
\item \label{ass:5} It holds that
	\begin{enumerate}
		\item \label{ass:BSSa}
			$g(x) = x^{\alpha}L_g(x)$, where $L_g$ is slowly varying at zero.
		\item \label{ass:BSSb}
			$g'(x) = x^{\alpha-1}L_{g'}(x)$, where $L_{g'}$ is slowly varying at zero, and, for any $\epsilon>0$, we have $g'\in L^2((\epsilon,\infty))$. Also, for some $a>0$, $|g'|$ is non-increasing on the interval $(a,\infty)$.\footnote{Again following \cite{BHLP16}, in the case $\alpha = 0$ an alternative assumption is adopted: (\ref{ass:BSSb}') $g'(x) = L_{g'}(x)$, where $L_{g'}$ is as in (\ref{ass:BSSb}).}
		\item
			For any $t>0$,
			\begin{align*}
				F_t := \int_1^{\infty}|g'(x)|^2 \sigma_{t-x}^2 dx < \infty.
			\end{align*}
	\end{enumerate}
\end{enumerate}

The kernel function gives the $\BSS$ framework great flexibility. A particularly useful kernel function which has been applied in a number of studies, e.g. \cite{BNBV13} and \cite{bennedsen15}, is the so-called \emph{gamma kernel}.
\begin{example}[$\Gamma$-$\BSS$ process]\label{ex:gamma}
Let $g$ be the gamma kernel, i.e. $g(x) = x^{\alpha} e^{-\lambda x}$ for $\alpha \in (-1/2,1/2)$ and $\lambda>0$. The resulting process 
\begin{align*}
X_{t}=\int_{-\infty }^{t}(t-s)^{\alpha} e^{-\lambda (t-s)}\sigma_s dW_{s}, \quad t \geq 0,
\end{align*}
is called the (volatility modulated) $\Gamma$-$\BSS$ process. It is not hard to show that this process fulfills assumptions (\ref{ass:1})--(\ref{ass:99}) and (\ref{ass:5}), see Example 2.3. in \cite{BLP15}.
\end{example}



\begin{remark}
In \cite{BLP15} it was shown that $\BSS$ processes satisfying (\ref{ass:1})--(\ref{ass:3}), (\ref{ass:4}), and (\ref{ass:5}) will have the same fractal and continuity properties as their Gaussian counterparts: for such a $\BSS$ process Proposition \ref{lem:holder} continues to hold. In other words, $X$ will have a modification with H\"older continuous trajectories of order $\phi$ for all $\phi \in (0,\alpha + 1/2)$. 
\end{remark}

\subsection{Extension to processes with non-stationary increments}\label{sec:nonst}
When the increments of $X$ are non-stationary  an approach similar to the one in \cite{BLP15} can be adopted as follows. Define the time-dependent variogram 
\begin{align*}
\gamma_2(h,t) := \E[|X_{t+h}-X_t|^2], \quad h, t \in \R,
\end{align*}
and, analogously to \eqref{eq:scaling1}, assume that
\begin{align}\label{eq:scalingNS}
\gamma_2(h,t) = C_{2,t}|h|^{2\alpha+1}L(h), \quad t>0, \quad h\in \R,
\end{align}
where again $C_{2,t} >0$, $\alpha \in \left(-\frac{1}{2},\frac{1}{2}\right)$, and $L$ is a slowly varying function at zero. The methods considered in this paper applies  -- mutatis mutandis -- also to such processes. An example is the truncated Brownian semistationary process.

\begin{example}[Truncated $\BSS$ process, \cite{BLP15}]
Let
\begin{align*}
X_t = X_0 +  \int_{0}^t g(t-s)\sigma_s dW_s, \quad t \geq 0,
\end{align*}
where $X_0 \in \R$, $W$ is a Brownian motion, and $\sigma$ a stochastic volatility process. \cite{BLP15} call such a process a \emph{truncated} $\BSS$ ($\mathcal{TBSS}$) process. When	 $X$ satisfies (\ref{ass:1})--(\ref{ass:3}) and (\ref{ass:5}), \cite{BLP15} show that $\alpha$ is indeed the fractal index of $X$, in the sense of $\gamma_2(h,t)$ satisfying \eqref{eq:scalingNS}. We note that processes similar to the $\mathcal{TBSS}$ process (with $\sigma_t = 1$ for all $t$) have recently been proposed as models of stochastic log-volatility of financial assets, e.g., \cite{GJR14,BFG15}.
\end{example}

\subsection{Summary of assumptions}
Above we introduced a number of processes, differing in important ways, most notably through their distributional properties. In spite of these differences, the results presented in this paper will apply equally to all of them. To ease notation, we briefly summarize the assumptions here. 

The first set of assumptions is required for consistency of the estimator of the fractal index $\alpha$.

\begin{enumerate}[label=(LLN),ref=LLN]
\item \label{ass:LLN} Suppose that one of the following holds:
	\begin{enumerate}
		\item \label{ass:lln1}
			$X$ is Gaussian satisfying (\ref{ass:1})--(\ref{ass:3}) for an $\alpha \in (-1/2,1/2)$.
		\item \label{ass:lln2}
			$X$ is defined by \eqref{eq:Xsv} where $G$ satifies (\ref{ass:1})--(\ref{ass:3}) for an $\alpha \in (-1/2,1/2)$ and $\sigma$ has finite $q$-variation for all $q < \frac{1}{1/2-\alpha}$.
		\item \label{ass:lln3}
			$X$ is a $\BSS$ process, defined by \eqref{eq:Xbss}, satisfying (\ref{ass:1})--(\ref{ass:3}) for an $\alpha \in (-1/2,1/2)$ and with kernel function $g$ satisfying (\ref{ass:5}). 
	\end{enumerate}
\end{enumerate}

The second set of assumptions is required for asymptotic normality of the estimator of the fractal index $\alpha$.

\begin{enumerate}[label=(CLT),ref=CLT]
\item \label{ass:CLT} Suppose that one of the following holds:
	\begin{enumerate}
		\item \label{ass:clt1}
			$X$ is Gaussian satisfying (\ref{ass:LLN}) for an $\alpha \in (-1/2,1/4)$, as well as (\ref{ass:99}).
		\item \label{ass:clt2}
			$X$ is defined by \eqref{eq:Xsv} satisfying (\ref{ass:LLN}) for an $\alpha \in (-1/2,1/4)$, as well as (\ref{ass:99}). The process $\sigma$ additionally fulfills (\ref{ass:4}).
		\item \label{ass:clt3}
			$X$ is a $\BSS$ process, defined by \eqref{eq:Xbss},  satisfying (\ref{ass:LLN}) for an $\alpha \in (-1/2,1/4)$, as well as (\ref{ass:99}). The process $\sigma$ additionally fulfills (\ref{ass:4}).
	\end{enumerate}
\end{enumerate}

\begin{remark}\label{rem:rates}
As seen from the assumptions, the central limit theorems will not be applicable for $\alpha \geq 1/4$. In fact, a central limit theorem do hold in this case, but with a different convergence rate and limiting distribution from what we derive below. When $\alpha = 1/4$, the convergence rate is $\sqrt{\frac{n}{\log n}}$ and the limiting distribution is zero-mean Gaussian with an asymptotic variance different from when $\alpha <1/4$. When $\alpha > 1/4$ the convergence rate is $n^{1-2\alpha}$ and the limiting distribution is of the Rosenblatt type, see \cite{taqqu79}.  If one is interested in the range $\alpha \in [1/4,1/2)$ and desire asymptotic normality results similar to what we have below, we recommend using gaps between the observations as in \cite{CHPP13} Remark 4.4; the downside of this approach is that one is forced to throw away observations. Given the results presented below, filling in the details of this approach is straight forward, albeit notationally cumbersome. Since the case of very smooth processes, i.e. $\alpha \geq 1/4$, seems of limited practical value, we do not pursue this further here.
\end{remark}

\section{Semiparametric estimation of, and inference on, the fractal index}\label{sec:bootstrap}
Consider $n$ equidistant observations $X_{1/n}, X_{2/n}, \ldots, X_{1}$ of the stochastic process $X$, observed over a fixed time interval, which we without loss of generality take to be the unit interval, so that the time between observations is $\frac{1}{n}$. As $n \rightarrow \infty$, this gives rise to the so-called \emph{in-fill asymptotics}. In what follows, suppose that the process $X$ satisfies the assumptions (\ref{ass:1})--(\ref{ass:3}).

When $X$ is Gaussian, it holds, by standard properties on the (absolute) moments of the Gaussian distribution and \eqref{eq:scaling1}, that
\begin{align}\label{eq:gam_rel}
\gamma_p(h;X) = C_p |h|^{(2\alpha+1)p/2} L_p(h), \quad h \in \R,
\end{align}
where $p>0$, the function $L_p(h) := L(h)^{p/2}$ is slowly varying at zero, and $C_p>0$ is a constant. This motivates the regression
\begin{align}\label{eq:reg_a}
\log \hat{\gamma}_p(k/n;X) = c_p +a \log|k/n| + U_{k,n} + \epsilon_{k,n}, \quad k = 1, 2, \ldots, m,
\end{align}
where $m \in \N$ is a bandwidth parameter, 
\begin{align*}
c_p = \log C_p, \quad a = \frac{(2\alpha+1)p}{2}, \quad U_{k,n} = \log \left( \frac{\hat{\gamma}_p(k/n;X)}{\gamma_p(k/n;X)} \right), \quad \textnormal{ and} \quad \epsilon_{k,n} =\log L_p(k/n).
\end{align*}
The variogram $\gamma_p$ is estimated straightforwardly as 
\begin{align}\label{eq:gam_hat}
\hat{\gamma}_p(k/n;X) := \frac{1}{n-k} \sum_{i=1}^{n-k} |X_{\frac{i+k}{n}} - X_{\frac{i}{n}}|^p, \quad k\geq 1.
\end{align} 
The OLS estimator of the parameter $a$ is naturally
\begin{align*}
\hat{a}_{OLS} = \frac{1}{x_m^T x_m} x_m^T \log \hat{\gamma}_p^m,
\end{align*} 
with ``T" denoting the transpose of a vector and $x_m$ being the $m\times 1$ vector
\begin{align*}
x_m := \left(\log 1 - \overline{\log m}, \log 2 - \overline{\log m}, \ldots , \log m - \overline{\log m}\right)^T, \qquad \overline{\log m} := \frac{1}{m}\sum_{k=1}^m \log k,
\end{align*}
while
\begin{align*}
\log \hat{\gamma}_p^m := \left(\log \hat{\gamma}_p(1/n;X),\log \hat{\gamma}_p(2/n;X), \ldots , \log \hat{\gamma}_p(m/n;X)\right)^T.
\end{align*}
Given an estimate $\hat{a}_{OLS}$ of $a$, our estimate of the fractal index is 
\begin{align}\label{eq:olsA}
\hat{\alpha} := \frac{\hat{a}_{OLS}}{p}-\frac{1}{2}.
\end{align}
This estimator is well known and much used in the literature, e.g. \cite{GS04,GJR14,BLP16}. The following proposition shows the consistency of the OLS estimator of $\alpha$.

\begin{proposition}\label{prop:cons}
Suppose (\ref{ass:LLN}) holds. Fix $p>0$, $m \in \N$, and let $\hat{\alpha} = \hat{\alpha}_{p,m}$ be the OLS estimator of $\alpha$ from \eqref{eq:olsA}. Now, 
\begin{align*}
\hat{\alpha} \stackrel{P}{\rightarrow} \alpha, \quad n \rightarrow \infty,
\end{align*}
where ``$\stackrel{P}{\rightarrow}$" refers to convergence in $\mathbb{P}$-probability.
\end{proposition}

A number of studies have considered the asymptotic properties of the OLS estimates coming from \eqref{eq:olsA}, e.g. \cite{CH94}, \cite{DH99}, and \cite{coeurjolly01,coeurjolly08}. For a brief summary of this literature, see \cite{GSP12}, Section 3.1. The following theorem presents the details in the context of this paper.

\begin{theorem}\label{th:fbm}
Suppose (\ref{ass:CLT}) holds. Fix $p>0$, $m \in \N$, and let $\hat{\alpha} = \hat{\alpha}_{p,m}$ be the OLS estimator of $\alpha$ from \eqref{eq:olsA}. If (\ref{ass:clt2}) or (\ref{ass:clt3}) holds, we require $\xi \cdot \min\{p,1\} > 1/2$, cf. assumption (\ref{ass:4}). Now, as $n \rightarrow \infty$,
\begin{align*}
\sqrt{n}(\hat{\alpha} - \alpha) \stackrel{st}{\rightarrow} Z_p \cdot S_p, \qquad Z_p \sim N\left(0, \sigma_{m,p}^2\right),
\end{align*}
where
\begin{align*}
\sigma^2_{m,p} = \frac{x_m^T \Lambda_p x_m}{ (x_m^T x_m)^2 p^2},
\end{align*}
and $\Lambda_p = \{ \lambda_p^{k,v} \}_{k,v=1}^m$ is a real-valued $m \times m$ matrix with entries
\begin{align}\label{eq:lamLimTh}
\lambda_p^{k,v} =   \lim_{n\rightarrow \infty} n \cdot Cov \left(\frac{\hat{\gamma}_p(k/n;B^H)}{ \gamma_p(k/n;B^H)},\frac{\hat{\gamma}_p(v/n;B^H)}{ \gamma_p(v/n;B^H)}  \right), \quad k,v = 1, 2, \ldots, m,
\end{align}
where $\hat{\gamma}_{\cdot}(\cdot;B^H,\cdot)$ is given by \eqref{eq:gam_hat} and $B^H$ is a fractional Brownian motion with Hurst parameter $H = \alpha + \frac{1}{2}$.

Further, if (\ref{ass:clt1}) holds, then
\begin{align*}
S_p = 1,
\end{align*}
while if (\ref{ass:clt2}) or (\ref{ass:clt3}) holds, then
\begin{align}\label{eq:Sp}
S_p = \frac{ \sqrt{ \int_0^1 \sigma_s^{2p} ds} }{ \int_0^1 \sigma_s^p ds }.
\end{align}
Above ``st" denotes stable convergence (in law), see e.g. \cite{renyi63}.
\end{theorem}

\begin{remark}
The limit in \eqref{eq:lamLimTh} exists for $k,v = 1, \ldots, m$ by \cite{BM83}, Theorem 1. See also Remark 3.3. in \cite{CHPP13}.
\end{remark}

Perhaps surprisingly, Theorem \ref{th:fbm} shows that the asymptotic distribution of the OLS estimator does not depend on the precise structure of the underlying process $X$, but only on the value of the fractal index $\alpha$, through the correlation structure of the increments of a fractional Brownian motion (fBm) with Hurst index $H = \alpha + 1/2$, and possibly the ``heteroskedasticity factor" $S_p$. The reason for this is that the small scale behavior of a process $X$ fulfilling assumption (\ref{ass:1}), will have the same small scale behavior as increments of the fBm. To see this, write
\begin{align}
r_n(j) &:= Corr\left( X_{\frac{j+1}{n} } - X_{\frac{j}{n}},X_{\frac{1}{n} } - X_{0} \right) \nonumber \\
	&= \frac{\gamma_2((j+1)/n;X) -2 \gamma_2(j/n;X) + \gamma_2((j-1)/n;X)}{2\gamma_2(1/n;X)} \nonumber \\
	&\rightarrow \frac{1}{2} \left( |j+1|^{2\alpha+1} -2 |j|^{2\alpha+1} +|j-1|^{2\alpha+1} \right), \quad n\rightarrow \infty, \label{eq:rnlim}
\end{align}
by assumption (\ref{ass:1}) and the properties of slowly varying functions. We recognize \eqref{eq:rnlim} as the correlation function of the increments of an fBm with Hurst index $H =\alpha+1/2$. As shown in the proof of Theorem \ref{th:fbm}, this will imply that the asymptotic variance of the estimator, $\sigma_{m,p}^2$, is the same for \emph{all} Gaussian processes fulfilling assumptions (\ref{ass:1})--(\ref{ass:99}), including the fBm. However, as the theorem also shows, the asymptotic distribution proves to be slightly different when we consider \emph{conditionally} Gaussian processes. In this case, the stochastic volatility component $\sigma$ introduces heteroskedasticity, which results in the extra factor $S_p$ in the central limit theorem. To make inference feasible in practice, we need to estimate this factor. For this, define
\begin{align}\label{eq:Sp_hat}
\widehat{S_p} = \frac{ \sqrt{m_{2p}^{-1}\hat{\gamma}_{2p}(1/n;X)}}{m_p^{-1}\hat{\gamma}_p(1/n;X)}, \qquad m_s := \frac{2^{s/2}}{\sqrt{\pi}}\Gamma\left(\frac{s+1}{2}\right),  \qquad p, s > 0,
\end{align}
where $\hat{\gamma}_{\cdot}$ is given in \eqref{eq:gam_hat}. We can prove the following.
\begin{proposition}\label{prop:S2fbm}
(i) Suppose (\ref{ass:lln1}) holds. Let $p>0$. Now,
\begin{align*}
\widehat{S_p}  \stackrel{P}{\rightarrow} 1, \quad n\rightarrow \infty.
\end{align*}

(ii) Suppose (\ref{ass:lln2}) or (\ref{ass:lln3}) holds. Let $p>0$. Now,
\begin{align*}
\widehat{S_p}  \stackrel{P}{\rightarrow} S_p, \quad n \rightarrow \infty,
\end{align*}
where $S_p$ is given in \eqref{eq:Sp}.
\end{proposition}
Proposition \ref{prop:S2fbm} shows that $\widehat{S_p}$ of \eqref{eq:Sp_hat} is a suitable estimator for our purpose: when $X$ is Gaussian, the factor is asymptotically irrelevant, while when $X$ is non-Gaussian (volatility modulated) it provides the correct normalization. This justifies including the factor $\widehat{S_p}$, whether or not one believes the data is Gaussian, at least when any potential non-Gaussianity is volatility induced. In fact, the following corollary is a straightforward consequence of Theorem \ref{th:fbm}, Proposition \ref{prop:S2fbm}, and the properties of stable convergence; the corollary has obvious applications to feasibly conducting inference and making confidence intervals for $\alpha$.

\begin{corollary}\label{cor:feas}
Suppose the assumptions of Theorem \ref{th:fbm} hold. Now,
\begin{align*}
\sqrt{n} \frac{\hat{\alpha} - \alpha}{ \widehat{S_p}\sqrt{ \sigma_{m,p}^2(\hat{\alpha})}} \stackrel{d}{\rightarrow} N(0,1), \quad n \rightarrow \infty,
\end{align*}
where ``d" denotes convergence in distribution and $\sigma_{m,p}^2(\hat{\alpha})$ denotes the asymptotic variance calculated using the estimate $\hat{\alpha}$.
\end{corollary}

\begin{remark}
When using Corollary \ref{cor:feas} for hypothesis testing, we recommend calculating $\sigma_{m,p}^2(\cdot)$ using the value of $\alpha$ under the null, instead of $\hat{\alpha}$.
\end{remark}

To apply the above results we need to calculate the factor $\sigma_{m,p}^2$, which boils down to calculating the entries of the matrix $\Lambda_p$ given in equation \eqref{eq:lamLimTh}. Unfortunately, this is only feasible when $p=2$ and becomes increasingly cumbersome as $m$ increases. \citep[The already tedious calculation for $p=m=2$ is given in][Appendix B.]{BHLP16}. For this reason, we recommend Monte Carlo estimation of $\sigma_{m,p}^2$; in fact, we suggest using the finite sample analogue of this factor. The procedure is detailed in Appendix \ref{app:simSig}; in the next section we present an example of the output, when we study the effect of the choice of bandwidth, $m$.

\subsection{Choosing the bandwidth parameter}\label{sec:band}

The choice of bandwidth parameter $m$ is, in general, an open problem. Standard practice in the literature is to set $m = 2$ \citep[][Section 2.3]{GSP12}. Indeed, \cite{CH94} argue that the bias of the estimator increases with $m$ and \cite{DH99} present simulation evidence for the optimal value, in terms of mean squared error, being $m=2$. Setting $m=2$ amounts to estimating $\alpha$ by drawing a straight line between only the two points closest to the origin, $\log \hat{\gamma}_p(1/n;X)$ and $\log \hat{\gamma}_p(2/n;X)$, when running the OLS regression in \eqref{eq:reg_a}. While tempting from a bias viewpoint, we conjecture that this can result in increased variance of the estimator, by relying on just two points in the regression. In what follows, we examine this in more depth. To be specific, we consider the effect that the bandwidth has on the estimator of the fractal index; first on the theoretical (finite sample) variance of $\hat{\alpha}$, as derived in Theorem \ref{th:fbm} (Figure \ref{fig:sig2_m}), and then on the finite sample bias and mean squared error of the estimator when applied to simulated paths of the various processes of Table \ref{tab:ex} (Figure \ref{fig:br}).  For these investigations, we consider both $\alpha = -0.20$ (rough case) and $\alpha = 0.20$ (smooth case).

\begin{figure}[!t] 
\centering 
\includegraphics[width=\textwidth]{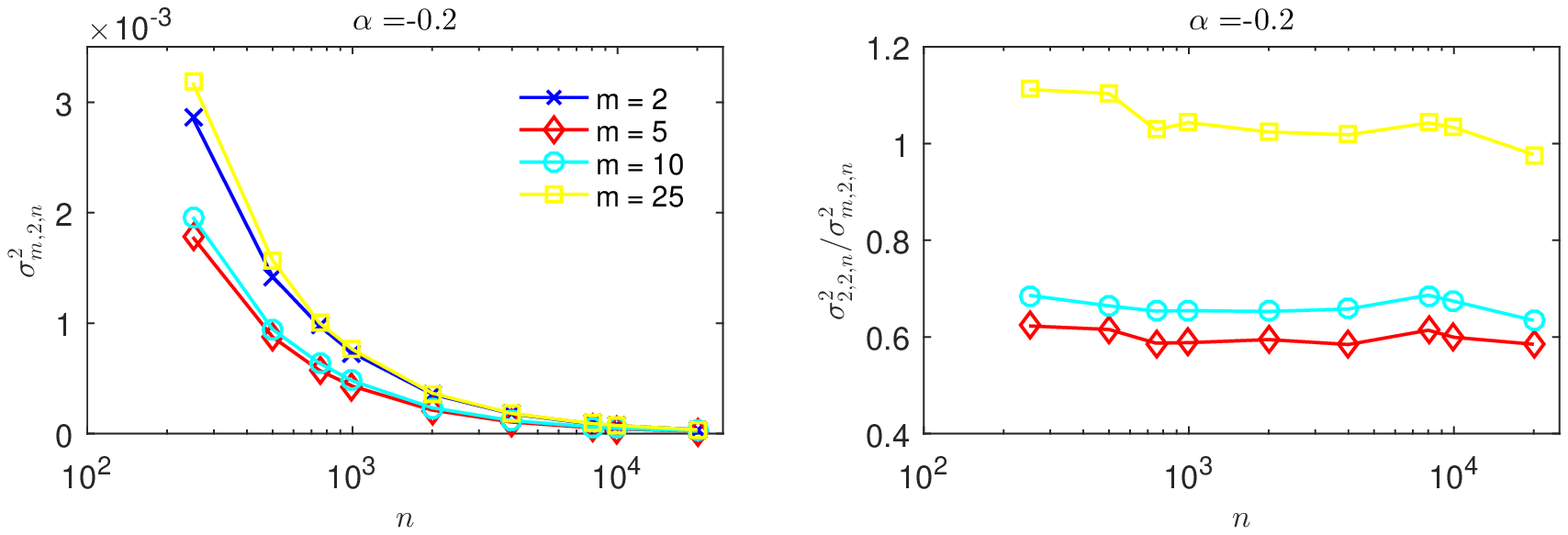}\\ 
\vspace{0.25cm}
\includegraphics[width=\textwidth]{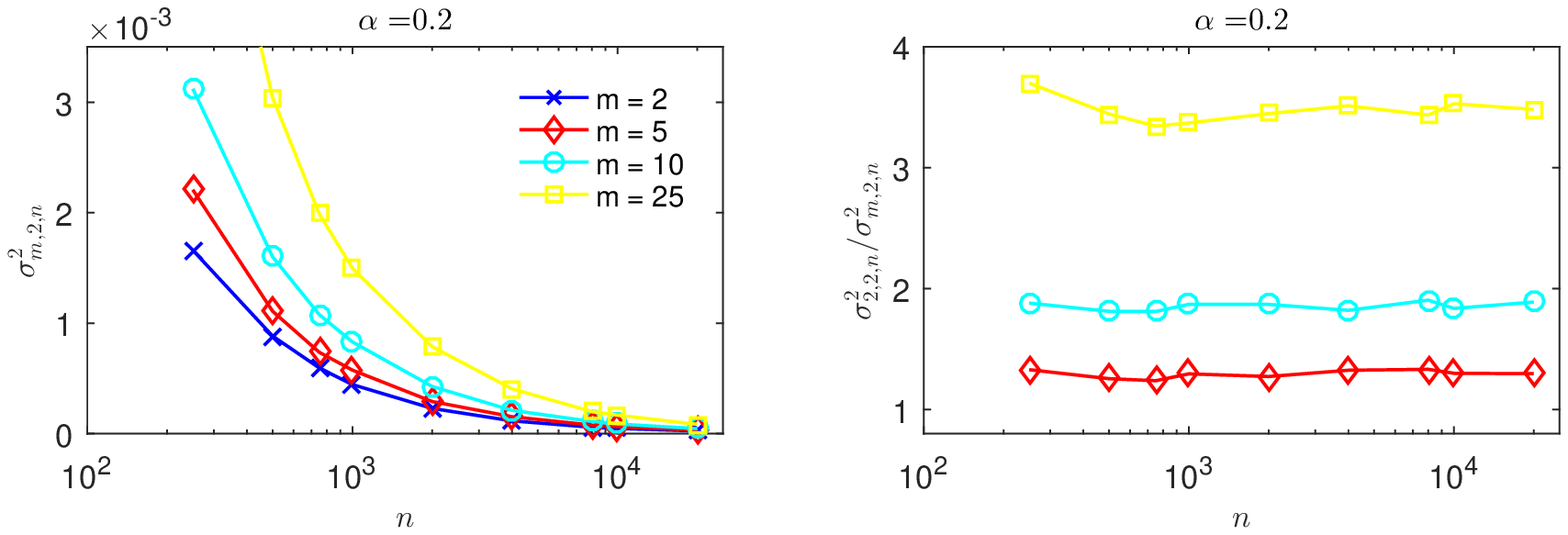}
\caption{Monte Carlo approximation ($B = 10\ 000$ replications) of the finite sample analogue of the variance of $\hat{\alpha}$, $\sigma^2_{m,p,n} \approx n^{-1}\sigma^2_{m,p}$, cf. Theorem \ref{th:fbm}. The true value of $\alpha$ is indicated above the plots. See Appendix \ref{app:simSig} for details of the calculations.}
\label{fig:sig2_m}
\end{figure}

Figure \ref{fig:sig2_m} studies the effect that the choice of bandwidth has on the variance of the estimator of $\alpha$: we plot the approximation of the finite sample variance of $\hat{\alpha}$, $\sigma_{m,p,n}^2$, which is approximately equal to $n^{-1}\sigma_{m,p}^2$, cf. Theorem \ref{th:fbm}. From the figure, we see that the choice of bandwidth indeed has an effect on the variance of the OLS estimator of $\alpha$. Interestingly, the effect is very different in the rough case, as compared to the smooth case. In the former, it is evident from the top left plot of Figure \ref{fig:sig2_m}, that the variance is minimized by an intermediate value of $m$ such as $m=5$ or $m = 10$. To further investigate this, the top right plot shows the ratio between the finite sample variance when $m = 2$ and when $m \in \{5,10,25\}$. Numbers less than one indicate that the variance of the estimator with $m=2$ is greater than the variance of the corresponding estimator with $m>2$, and vice versa. These ratios seem quite stable as a function of sample size $n$ and it is evident that, from a variance stand point, it is preferable to choose an intermediate $m>2$ --- indeed, the variance of the estimator is reduced by approximately $40\%$ when going from $m = 2$ to $m = 5$. These conclusions get turned on their heads when we consider the smooth case, $\alpha = 0.20$, in the bottom row: here it seems that $m=2$ is optimal.


\begin{figure}[!t] 
\centering 
\includegraphics[width=\textwidth]{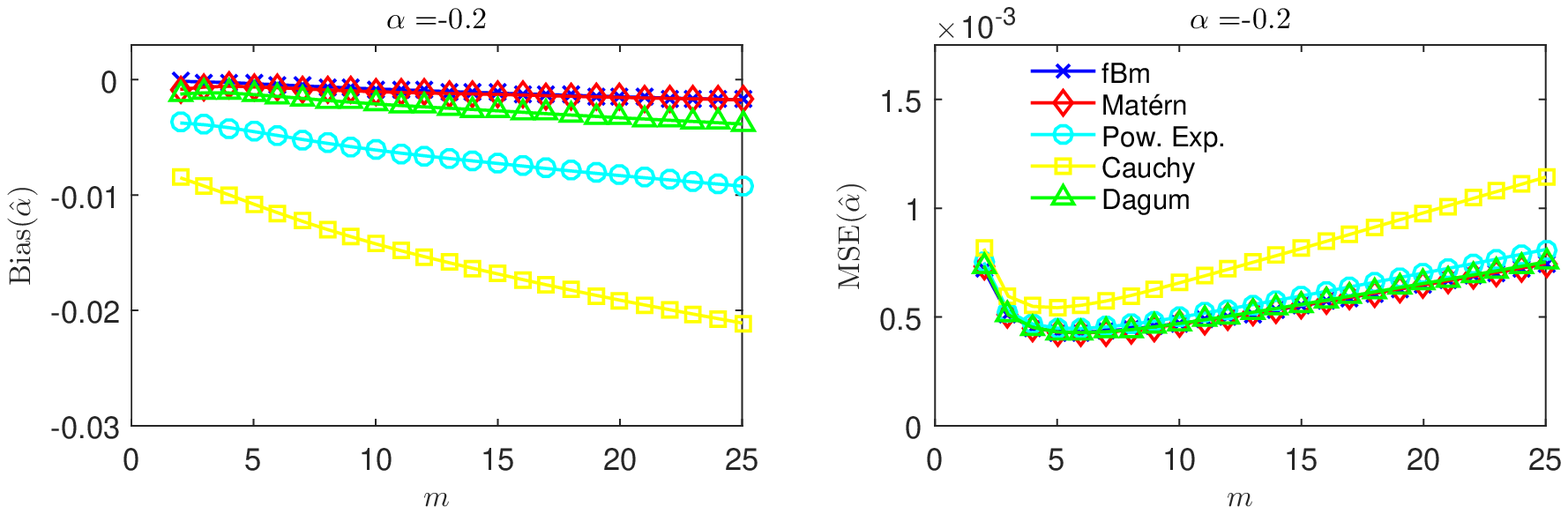} \\ 
\vspace{0.25cm}
\includegraphics[width=\textwidth]{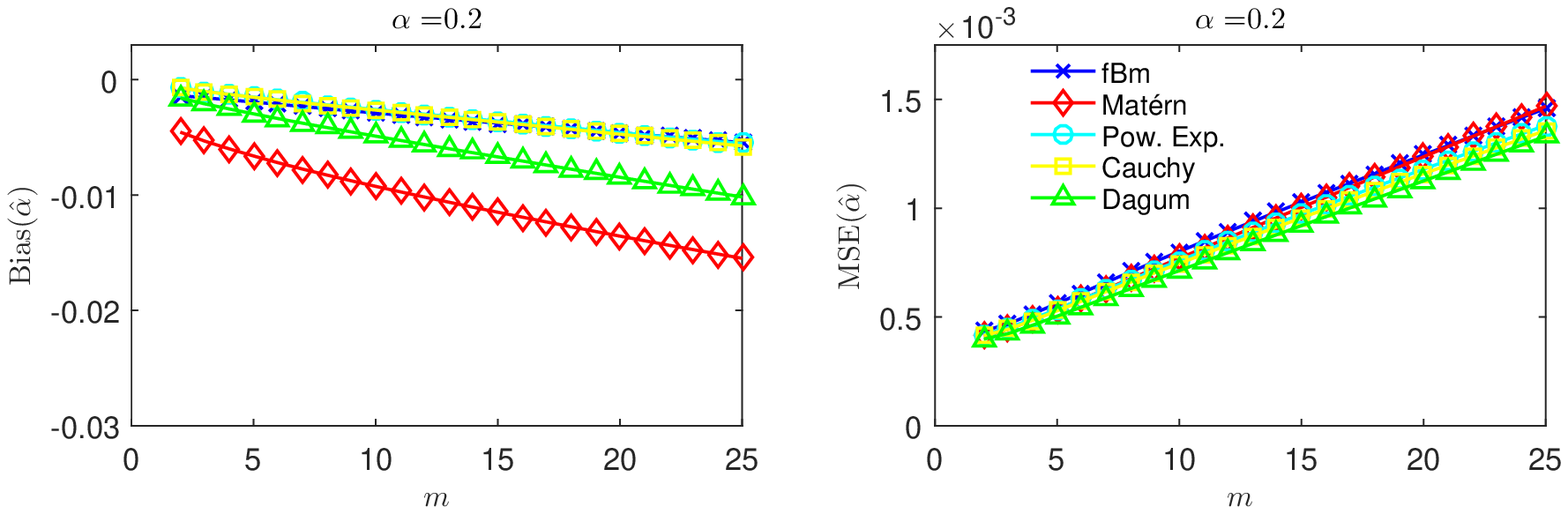} 
\caption{Monte Carlo approximation ($B = 10\ 000$ replications) of the finite sample bias (left) and mean squared error (MSE, right) of $\hat{\alpha}$, as function of the bandwidth $m$, for the processes of Table \ref{tab:ex}. We set $p = 2$ and $n = 1\ 000$. The parameter values are given in the text.}
\label{fig:br}
\end{figure}

We further investigate this through simulations as in \cite{DH99}: Figure \ref{fig:br} plots the bias (left) and mean squared error (right) of the estimator \eqref{eq:olsA}, as a function of bandwidth $m$, for the five parametric processes of Table \ref{tab:ex}. To calculate the finite sample bias and mean squared error of the estimator, we simulate $B = 10\ 000$ instances of each process, each with $n = 1\ 000$ observations; the true value of the fractal index in this exercise is $\alpha = -0.2$ (top row) and $\alpha = 0.2$ (bottom row). The scale parameter is set to $\beta = 1$. For the Cauchy and Dagum processes we additionally set $\tau = 1$ and $\tau = 0$, respectively. When looking at the rough case, $\alpha = -0.2$, the same conclusion as above emerges: even though the bias do increase, as expected, with increasing $m$, it is clear that the mean squared error is minimized for an $m>2$. In this case, i.e. for these parameter values and this sample size, the minimum is attained between $m = 5$ and $m=10$ for all five processes. We again conclude that an intermediate value for the bandwidth is preferable in finite samples when $\alpha < 0$. The smooth case, $\alpha = 0.2$, also matches what we found above: indeed, we find that both bias and mean squared error increase with increasing $m$, so here $m=2$ seems optimal.

In conclusion, the evidence of this section suggests that when the underlying process is rough, the optimal choice of bandwidth is some $m>2$ and we recommend an intermediate value such as $m=5$. In contrast, when the process is smooth, $m=2$ is preferable. Although setting $m = 2$ seems to be accepted practice in the literature, we believe that the rough case of $\alpha <0$ is arguably more relevant in empirical applications. For this reason we suggest using an intermediate value for the bandwidth parameter, unless one has reason to believe the underlying data to be smooth.\footnote{A fortiori, simulations not reported here suggest that an intermediate value for $m$ is also preferable when $\alpha \approx 0$.}

\subsection{Asymptotic theory in the presence of additive noise}\label{sec:assNoise}
Consider now the situation, where the observations of $X$, satisfying (\ref{ass:1})--(\ref{ass:3}), are contaminated by additive noise; that is, instead of observing $X$, we observe the process $Z$, given by
\begin{align}\label{eq:Z}
Z_{j/n} := \mu + X_{j/n} + u_j, \quad j = 1, 2, \ldots, n,
\end{align}
where $\mu \in \R$ is a constant and $u = \{u_j\}_{j=1}^n$ is a Gaussian iid noise sequence with mean zero and variance $\sigma_u^2 := Var(u_1) \geq 0$. (When $\sigma_u^2 = 0$ we mean that the noise is absent from the observations)

Since we observe $Z$, and not $X$, what is relevant for us is the ``contaminated", or ``noisy", variogram, i.e. the variogram of the observation process $Z$:
\begin{align}\label{eq:varioZ}
\gamma_2(h;Z) = \E [ |Z_{t+h} - Z_t|^2] = \gamma_2(h;X) + 2\sigma_u^2 = h^{2\alpha+1}L(h) + 2\sigma_u^2, \quad h \in \R,
\end{align}
where the last equality follows from Assumption (\ref{ass:1}). From this we see that when $\sigma_u^2 >0$, $\log \gamma_2(h;Z)$ will not be linear in $\log h$, hence the estimator \eqref{eq:olsA} of $\alpha$ will not be applicable; in fact, it is not hard to show that this estimator will be \emph{downwards} biased in the presence of noise, i.e. when applied to $\gamma_p(\cdot;Z)$. In fact, the following is true.
\begin{proposition}\label{prop:incons}
Suppose  that the observations of a process $Z$ are given by \eqref{eq:Z} with $\sigma_u^2>0$, where $X$ satisfies assumption (\ref{ass:LLN}). Fix $p>0$, $m \in \N$, and let $\hat{\alpha} = \hat{\alpha}_{p,m}$ be the OLS estimator of $\alpha$ from \eqref{eq:olsA} using the contaminated version of the empirical variogram $ \hat{\gamma}_{p}(\cdot;Z)$ in place of $\hat{\gamma}_{p}(\cdot;X)$ in the regression \eqref{eq:reg_a}. Now,
\begin{align*}
\hat{\alpha} \stackrel{P}{\rightarrow} -1/2, \quad n \rightarrow \infty.
\end{align*}
\end{proposition}

Proposition \ref{prop:incons} shows that if the data are contaminated by noise, then estimates of the parameter $\alpha$ will be biased downwards towards $-1/2$, i.e. the lowest permissible value for $\alpha$. In other words, if the data are contaminated by noise, then the estimator of $\alpha$ considered above, will lead one to conclude that the data are more rough than what is actually the case for the underlying process $X$. This is an important point to note for the practitioner: when finding evidence of roughness (i.e. $\alpha < 0$) in data, it is crucial to consider whether this is due to an intrinsic property of the underlying data generating mechanism or whether it could simply be the product of noise, e.g. measurement noise.

Fortunately, it is possible to account for the noise when estimating $\alpha$ to arrive at a consistent estimator. For instance, \cite{BLP16} suggest a noise-robust estimator based on a non-linear least squares regression -- however, this estimator does not allow for the slowly varying function $L$ and requires the interval over which the process is observed to grow. Presently, therefore, we propose an alternative noise-robust estimator which is valid in our in-fill asymptotics setup and again relies on a simple OLS regression.

First, for an integer $\kappa \geq 2,$ define the function
\begin{align*}
f_p(h;Z,\kappa) := \gamma_p(\kappa h;Z)^{2/p} - \gamma_p(h;Z)^{2/p}, \quad  h\in \R.
\end{align*}
From \eqref{eq:varioZ} and assumption (\ref{ass:1}), we have
\begin{align}\label{eq:f_rel}
f_p(h;Z,\kappa) =  C_p^{2/p}|h|^{2\alpha+1}L_p^*(h;\kappa), \quad  h\in \R,
\end{align}
where it is easy to show that the function
\begin{align*}
L_p^*(h;\kappa) :=  \left(\kappa^{2\alpha+1}L(\kappa h)^{2/p} -L(h)^{2/p}\right), \quad  h\in \R,
\end{align*}
is slowly varying at zero. From this, it is clear that  the logarithm of $f_p(h;Z,\kappa)$ is -- up to the slowly varying function $L_p^*$ -- linear in $\log h$. This motivates a linear regression as the one in \eqref{eq:reg_a} with $\log \hat{f}_p$ in place of $\log \hat{\gamma}_p$: 
\begin{align}\label{eq:olsZ}
\log \hat{f}_p(k/n;Z,\kappa)  = b^* + a^* \log |k/n| + U^*_{k,n} + \epsilon^*_{k,n}, \quad k = 1, 2, \ldots, m,
\end{align}
where 
\begin{align*}
\hat{f}_p(k/n;Z,\kappa) := \hat{\gamma}_p(\kappa k/n;Z)^{2/p} - \hat{\gamma}_p(k/n;Z)^{2/p}
\end{align*}
is the empirical estimate of the function $f$, which is feasible to calculate from the observations $Z_{j/n}$.  Define the noise robust estimate of $\alpha$ as 
\begin{align}\label{eq:a_star}
\hat{\alpha}^* := \frac{\hat{a}_{OLS}^*}{2} - \frac{1}{2}, 
\end{align}
where $\hat{a}^*_{OLS}$ is the OLS estimate of $a^*$ from the linear regression \eqref{eq:olsZ}, analogous to \eqref{eq:olsA} with $\hat{f}_p$ in place of $\hat{\gamma}_p$. We can prove the following.

\begin{proposition}\label{prop:robust_cons}
Suppose  that the observations of a process $Z$ are given by \eqref{eq:Z} with $\sigma_u^2\geq 0$, where $X$ satisfies assumption (\ref{ass:LLN}). Fix $p>0$, $m \in \N$, and let $\hat{\alpha}^* = \hat{\alpha}^*_{p,m}$ be the OLS estimator of $\alpha$ from \eqref{eq:a_star}. Now,
\begin{align*}
\hat{\alpha}^* \stackrel{P}{\rightarrow} \alpha, \quad n \rightarrow \infty.
\end{align*}
\end{proposition}

\begin{remark}
Proposition \ref{prop:robust_cons} allows for $\sigma_u^2 = 0$ i.e. for there to be no noise in the observations. In other words, the robust estimator is a consistent estimator of $\alpha$, also in the absence of noise.
\end{remark}

In Figure \ref{fig:robust} we illustrate the use of Proposition \ref{prop:robust_cons} by calculating the bias and root mean squared error (RMSE) of the two OLS estimators given in \eqref{eq:olsA} and \eqref{eq:a_star}, when applied to a process $Z$ with $\sigma_u^2>0$. The details are provided in the caption of the figure. The former estimator is not robust to the noise in $Z$, while the latter estimator is per Proposition \ref{prop:robust_cons}. It is clear how this manifests itself in a large bias in the OLS estimator \eqref{eq:olsA}. In fact, although the true value of the fractal index of the underlying process is $\alpha = -0.20$, the mean OLS estimates coming from the non-robust estimator is $-0.4608$, i.e. almost at the lowest permissible value of $-1/2$. This is of course a consequence of Proposition \ref{prop:incons}. In contrast, the robust estimator  \eqref{eq:a_star} proposed in this section is practically unbiased for most values of the parameter $\kappa$, at least when $\kappa \geq 4$.

Although the results of this section hold for all integer $\kappa \geq 2$, the actual finite sample performance of the results can be quite sensitive to this tuning parameter, as also witnessed in Figure \ref{fig:robust}. The optimal choice of $\kappa$ seems to depend on the number of observations $n$ and the variance of the noise $\sigma_u^2$; an investigation into the exact way this is the case is beyond the scope of the present paper. In practice, we recommend that the researcher run some numerical experiments on simulated data under conditions similar to those of the practical experiment; simulation experiments such as the one in Figure \ref{fig:robust} for example. We provide an example of how one can construct such a simulation experiment to arrive at a reasonable value for $\kappa$ in Section \ref{sec:emp_fin}, where we apply the robust estimator to a time series of financial prices.

The next result provides the central limit theorem, as it relates to the robust estimator.

\begin{figure}[!t] 
\centering 
\includegraphics[width=\textwidth]{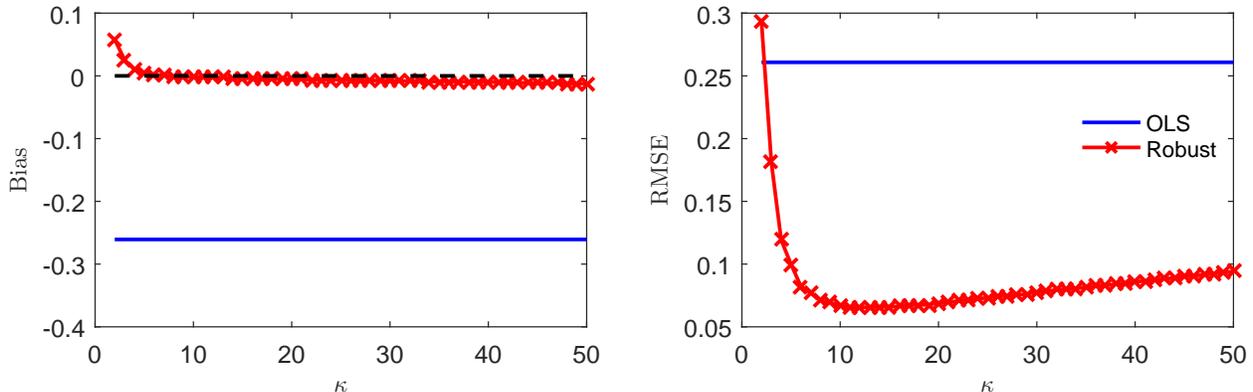}
\caption{Bias (left) and root mean squared error (RMSE, right) for the two OLS estimators \eqref{eq:olsA} and \eqref{eq:a_star}; blue line and red line with crosses, respectively. The bias and RMSE were calculated from $B = 10\ 000$ Monte Carlo simulations. The underlying data generating process for $X$ is $n = 2\ 500$ observations of an fBm with $\alpha = -0.20$, while $\mu = 1$ and $\sigma_u^2 = 0.05$ were used for the noise sequence. The bandwidth is $m = 5$. }
\label{fig:robust}
\end{figure}

\begin{theorem}\label{th:robust_clt}
Suppose that the observations of a process $Z$ are given by \eqref{eq:Z} with $\sigma_u^2\geq 0$, where $X$ satisfies assumption (\ref{ass:CLT}). Fix  $p > 0$, $m \in \N$, and let $\hat{\alpha}^* = \hat{\alpha}_{p,m}$  be the OLS estimator of $\alpha$ from \eqref{eq:a_star}. If (\ref{ass:clt2}) or (\ref{ass:clt3}) holds for $X$, we require $\xi \cdot \min\{p,1\} > 1/2$, cf. assumption (\ref{ass:4}). Now the following holds.
\begin{enumerate}[label=(\roman*)]
  \item \label{it:i} Let $\sigma_u^2 = 0$. As $n \rightarrow \infty$,
\begin{align*}
\sqrt{n}(\hat{\alpha}^* - \alpha) \stackrel{st}{\rightarrow} Z_p^* \cdot S_p, \qquad Z_p^* \sim N\left(0, \sigma_{m,p}^ {2,*}\right),
\end{align*}
where $S_p$  is as in Theorem \ref{th:fbm} and $\sigma^{2,*}_{m,p}$ is given in Appendix \ref{app:avar}.

\item \label{it:ii} Let $\sigma_u^2 > 0$. As $n \rightarrow \infty$,
\begin{align*}
\sqrt{n}\cdot |\hat{\alpha}^* - \alpha| \rightarrow \infty.
\end{align*}
\end{enumerate}

\end{theorem}

\begin{remark}
A feasible central limit theorem in the case of \ref{it:i} is straightforwardly constructed in the same way as in Corollary \ref{cor:feas}, including Monte Carlo estimation of $\sigma^{2,*}_{m,p}$.
\end{remark}

As shown in \ref{it:ii} of Theorem \ref{th:robust_clt}, the presence of the noise will unfortunately result in a variance of $\hat{\alpha}^*$, which decays slower than $\sqrt{n}$; indeed, the exact distribution of $\hat{\alpha}^*$ is difficult to derive and even harder to feasibly estimate. 

\subsubsection{A test for the presence of noise}\label{sec:testNoise}
Using the above, we can now construct a test for whether the observed time series $Z$ contains noise or not. To be specific, we are interested in testing the null hypothesis 
\begin{align}\label{eq:noise_test}
H_0: \sigma_u = 0 \qquad \textnormal{against the alternative} \qquad H_1: \sigma_u > 0.
\end{align}
Tests of this kind, in the context of time series of asset prices, were considered in \cite{ASX18}, where the authors develop a test for the presence of market microstructure noise in high frequency data. The test proposed here is similar in spirit to the test of \cite{ASX18} and in Section \ref{sec:emp_fin} we briefly consider testing for the presence of market microstructure in high frequency asset prices as well.

To device the test, we consider the difference between the robust estimator $\hat{\alpha}^*$ from \eqref{eq:a_star} and the usual (non-robust) estimator $\hat{\alpha}$ from \eqref{eq:olsA}. From Propositions \ref{prop:cons} and \ref{prop:robust_cons} is is immediately clear that under $H_0$
\begin{align*}
 \hat{\alpha}^* - \hat{\alpha} \stackrel{P}{\rightarrow} 0, \quad \textnormal{as } n\rightarrow \infty,
\end{align*}
while under $H_1$, Proposition \ref{prop:incons} additionally implies that
\begin{align*}
 \hat{\alpha}^* - \hat{\alpha} \stackrel{P}{\rightarrow} \alpha + 1/2 > 0, \quad \textnormal{as } n\rightarrow \infty.
\end{align*}
Analogously to Theorem \ref{th:robust_clt}, we can also prove the following.
\begin{theorem}\label{th:A}
Suppose the setup of Theorem \ref{th:robust_clt}. Now the following holds.
\begin{enumerate}[label=(\roman*)]
  \item \label{it:i} Let $\sigma_u^2 = 0$. As $n \rightarrow \infty$,
\begin{align*}
\sqrt{n}(\hat{\alpha}^* - \hat{\alpha}) \stackrel{st}{\rightarrow} Z_p^{**} \cdot S_p, \qquad Z_p^{**} \sim N\left(0, \sigma_{m,p}^ {2,**}\right),
\end{align*}
where $S_p$  is as in Theorem \ref{th:fbm} and $\sigma^{2,**}_{m,p}$ is given in Appendix \ref{app:avar}.

\item \label{it:ii} Let $\sigma_u^2 > 0$. As $n \rightarrow \infty$,
\begin{align*}
\sqrt{n}\cdot (\hat{\alpha}^* - \hat{\alpha}) \stackrel{P}{\rightarrow} \infty.
\end{align*}
\end{enumerate}
\end{theorem}
Define
\begin{align}\label{eq:A}
\hat{A}_n := \sqrt{n}\frac{\hat{\alpha}^* - \hat{\alpha}}{ \hat{S}_p \sqrt{ \sigma_{m,p}^{2,**}(\hat{\alpha}^*)}}.
\end{align}
The following corollary is a straightforward application of  Theorem \ref{th:A}.
\begin{corollary}\label{cor:A}
Let $\hat{A}_n$ be as in \eqref{eq:A}. Now,
\begin{enumerate}[label=(\roman*)]
\item
	Under $H_0$: $\hat{A}_n  \stackrel{d}{\rightarrow} N(0,1)$  as $n \rightarrow \infty$.
\item
	Under $H_1$: $\hat{A}_n  \rightarrow \infty$  as $n \rightarrow \infty$.
\end{enumerate}
\end{corollary}
The applicability of Corollary \ref{cor:A} for testing whether a fractal process is contaminated by noise is obvious. 

\begin{remark}
Above we have assumed that the noise sequence $u$ is Gaussian. However, one can show that all the results of sections \ref{sec:assNoise} and \ref{sec:testNoise} apply for general iid noise sequences with finite variance when $p=2$. In other words, if the Gaussian assumption on the noise sequence $u$ is not fulfilled -- or seems too restrictive -- then one should choose $p=2$ and go ahead and apply the results of these sections.
\end{remark}

\section{Simulation studies}\label{sec:sims}
To examine the finite sample properties of the central limit results presented above, we here conduct three small simulation studies and collect the results in Tables \ref{tab:1}--\ref{tab:3}. In each study we will let $X$ be an fBm with Hurst index $H = \alpha + 1/2$, for various values of $\alpha$, and simulate $n$ observations on the interval $[0,1]$. Additional information on the exacts simulation setups are given in the captions of the tables.

For a value $\alpha_0 \in (-1/2,1/4)$, Corollary \ref{cor:feas} allows us to test the null hypothesis
\begin{align}\label{eq:H00}
H_0: \alpha = \alpha_0 \qquad \textnormal{against the alternative} \qquad H_1: \alpha \neq \alpha_0.
\end{align}
Panel A of Table \ref{tab:1} shows the empirical size -- i.e., the rejection rates of $H_0$ when $H_0$ is true -- of this test for various values of $n$ and $\alpha$, at a nominal $5\%$ level. In contrast, Panel B shows the empirical local power of the test, that is, the rejection rate of the null hypothesis
\begin{align}\label{eq:H0n}
H_{0,n}: \alpha = \alpha_0 + n^{-1/2} \qquad \textnormal{against the alternative} \qquad H_{1,n}: \alpha \neq  \alpha_0 + n^{-1/2},
\end{align}
when $\alpha_0$ is the true value of $\alpha$ used in the simulations of the fBm and $n$ is the number of observations. In other words, the value of $\alpha$ that we test for asymptotes towards the true value of $\alpha$ at the rate $n^{1/2}$. We see that both the empirical size and power properties of the test are quite good. In particular, the test is approximately correctly sized for $n\geq 20$. (We conjecture that some of the already small deviation from the nominal $5\%$ level is likely due to Monte Carlo simulation error.) Similarly, the empirical power is also adequate. We conclude that the test is very precise and can comfortably be used in practice, even for small sample sizes.

Table \ref{tab:2} contains the analogous results in the setup of Theorem \ref{th:robust_clt} \ref{it:i}. Note that we here set $\sigma_u^2 = 0$ so that we do not have any noise in the observations. We set $\kappa = 10$ for all sample sizes but this could likely be optimized by considering different values of $\kappa$ for each different value of $n$. In this way, the empirical properties of the test presented in Table \ref{tab:2} are conservative, in the sense that one could likely obtain even better properties by choosing $\kappa$ in a data-driven way, cf., e.g., the approach taken in the setup of Figure \ref{fig:robust}. When considering the local power in this case, we replace the rate $n^{1/2}$ by $n^{1/4}$, that is, in Panel B of Table \ref{tab:2} we are testing the null hypothesis
\begin{align}\label{eq:H0n2}
H_{0,n}: \alpha = \alpha_0 + n^{-1/4} \qquad \textnormal{against the alternative} \qquad H_{1,n}: \alpha \neq  \alpha_0 + n^{-1/4},
\end{align}
where $\alpha_0$ is the true value of $\alpha$ used in the simulations of the fBm. Again we conclude that the empirical properties of this test are good. We see that we need more observations here to get a properly sized test, than what was the case above, the reason being that the estimator now utilizes gaps of size $\kappa = 10$ between observations, when calculating the statistic $\hat{f}$. The empirical power properties also seem adequate, although the rate is now somewhat slower than what we saw above. All in all, though, the practical relevance of this hypothesis test is not as great as the one considered above: we do not posses the asymptotic theory in the case noisy case ($\sigma_u^2 > 0$) and in the non-nosiy case ($\sigma_u^2 = 0$) the setup of the regular OLS estimator is preferable for inference. The robust estimator is consistent, though, so for pointwise estimation, the estimator should still have great practical relevance.

Finally, Table \ref{tab:3} contains the results from the test \eqref{eq:noise_test}; that is, we test for the presence of noise in the observation using Corollary \ref{cor:A}. Now the true observations are $Z_i = \mu + X_i + u_i$, where the variance of $u_i$ is $\sigma_u^2 \geq 0$. In Panel A we again have the size of the test, which is where we set $\sigma_u^2 = 0$, i.e. where $H_0$ is true. In Panel $B$ we have the power properties, which is the rejection rates of $H_0$ when $\sigma_u^2 = 0.05$, i.e. when $H_1$ is true.\footnote{As can be seen in Table \ref{tab:3}, in this simulation study we consider fewer Monte Carlo replications than what we did in the two studies above; this is because we here, contrary to above, do not have a null value for $\alpha$ -- hence we need to use the estimated value, $\hat{\alpha}^*$, when calculating the asymptotic variance by Monte Carlo simulation (cf., Appendix \ref{app:simSig}), whereas we in the other two studies could use the null value when calculation the asymptotic variance of the estimator. The upshot is, that in the present study, we have a Monte Carlo simulation (of the asymptotic variance given the estimate $\hat{\alpha}^*$) inside a Monte Carlo simulation (of the empirical rejection rates of the hyptohesis test), making the computational burden rapidly increasing in the number of Monte Carlo replications. We therefore choose a lower value for this number in the third study. In practice, for a given time series, one only needs to run one Monte Carlo study (approximating the asymptotic variance given the estimate $\hat{\alpha}^*$) which is quite fast -- hence in practice, the test is still feasible to do at high accuracy, i.e. with many Monte Carlo replications in the estimation of the asymptotic variance.} We see that the test is slightly under-sized in most cases, although reasonably close to the nominal level. The power properties are reasonably good for $n\geq 800$ at least when $\alpha$ is not too negative; when $\alpha \approx -0.50$, the test apparently needs many observations to be able to reject the null when it is false. This is to be expected, however, considering Proposition \ref{prop:incons}. Overall, the empirical properties of the test are quite good and we remind the reader that they can be further improved by choosing $\kappa$ in a data-driven way, see sections \ref{sec:assNoise} and \ref{sec:emp_fin} for examples.

\begin{table}
\caption{\it Hypothesis testing using Corollary \ref{cor:feas}}
\begin{center}
\footnotesize
\begin{tabularx}{0.99\textwidth}{@{\extracolsep{\stretch{1}}}lcccccccc@{}} 
\multicolumn{9}{l}{Panel A: Size} \\
\toprule
$n$  & \multicolumn{2}{c}{$\alpha = -0.40$} & \multicolumn{2}{c}{$\alpha = -0.20$} & \multicolumn{2}{c}{$\alpha = 0$} & \multicolumn{2}{c}{$\alpha = 0.20$} \\ 
\cmidrule{2-9}
 & $p = 1$ & $p = 2$ & $p = 1$ & $p = 2$ & $p = 1$ & $p = 2$ & $p = 1$ & $p = 2$ \\ 
 $ 10 $ & $ 0.0611 $ & $ 0.0929 $ & $ 0.0453 $ & $ 0.0746 $ & $ 0.0561 $ & $ 0.0865 $ & $ 0.0920 $ & $ 0.1385 $\\
 $ 20 $ & $ 0.0515 $ & $ 0.0655 $ & $ 0.0469 $ & $ 0.0690 $ & $ 0.0575 $ & $ 0.0807 $ & $ 0.0681 $ & $ 0.0861 $\\
 $ 40 $ & $ 0.0506 $ & $ 0.0514 $ & $ 0.0485 $ & $ 0.0601 $ & $ 0.0553 $ & $ 0.0671 $ & $ 0.0580 $ & $ 0.0661 $\\
 $ 80 $ & $ 0.0525 $ & $ 0.0554 $ & $ 0.0491 $ & $ 0.0565 $ & $ 0.0507 $ & $ 0.0610 $ & $ 0.0500 $ & $ 0.0569 $\\
 $ 160 $ & $ 0.0527 $ & $ 0.0564 $ & $ 0.0540 $ & $ 0.0475 $ & $ 0.0520 $ & $ 0.0522 $ & $ 0.0492 $ & $ 0.0498 $\\
 $ 320 $ & $ 0.0505 $ & $ 0.0500 $ & $ 0.0486 $ & $ 0.0539 $ & $ 0.0546 $ & $ 0.0539 $ & $ 0.0462 $ & $ 0.0518 $\\
 $ 10000 $ & $ 0.0460 $ & $ 0.0537 $ & $ 0.0485 $ & $ 0.0464 $ & $ 0.0508 $ & $ 0.0528 $ & $ 0.0511 $ & $ 0.0495 $\\
\bottomrule 
\end{tabularx}
\begin{tabularx}{0.99\textwidth}{@{\extracolsep{\stretch{1}}}lcccccccc@{}} 
\multicolumn{9}{l}{Panel B: Local power} \\
\toprule
   & \multicolumn{2}{c}{$\alpha = -0.40$} & \multicolumn{2}{c}{$\alpha = -0.20$} & \multicolumn{2}{c}{$\alpha = 0$} & \multicolumn{2}{c}{$\alpha = 0.20$} \\ 
\cmidrule{2-9}
 & $p = 1$ & $p = 2$ & $p = 1$ & $p = 2$ & $p = 1$ & $p = 2$ & $p = 1$ & $p = 2$ \\ 
 $ 10 $ & $ 0.1505 $ & $ 0.2837 $ & $ 0.2346 $ & $ 0.3670 $ & $ 0.3947 $ & $ 0.5143 $ & $ 0.8135 $ & $ 0.9551 $\\
 $ 20 $ & $ 0.2311 $ & $ 0.3186 $ & $ 0.2515 $ & $ 0.3510 $ & $ 0.2740 $ & $ 0.3865 $ & $ 0.5465 $ & $ 0.7050 $\\
 $ 40 $ & $ 0.2692 $ & $ 0.3299 $ & $ 0.2591 $ & $ 0.3322 $ & $ 0.2608 $ & $ 0.3355 $ & $ 0.3336 $ & $ 0.4245 $\\
 $ 80 $ & $ 0.3171 $ & $ 0.3715 $ & $ 0.2691 $ & $ 0.3273 $ & $ 0.2703 $ & $ 0.3439 $ & $ 0.2463 $ & $ 0.3365 $\\
 $ 160 $ & $ 0.3452 $ & $ 0.3999 $ & $ 0.2744 $ & $ 0.3217 $ & $ 0.2750 $ & $ 0.3252 $ & $ 0.2285 $ & $ 0.2711 $\\
 $ 320 $ & $ 0.3693 $ & $ 0.4151 $ & $ 0.2793 $ & $ 0.3312 $ & $ 0.2681 $ & $ 0.3157 $ & $ 0.2048 $ & $ 0.2442 $\\
 $ 10000 $ & $ 0.4225 $ & $ 0.4701 $ & $ 0.2958 $ & $ 0.3322 $ & $ 0.2727 $ & $ 0.3132 $ & $ 0.2098 $ & $ 0.2292 $\\
\bottomrule 
\end{tabularx}
\end{center}
{\footnotesize \it Panel A: Empirical rejection rates of the test \eqref{eq:H00}. Panel B: Empirical rejection rates of the test \eqref{eq:H0n}. The nominal significance level is $0.05$; the numbers in the table are the average rejection rate of $10\ 000$ Monte Carlo replications; and we used $B = 10\ 000$ Monte Carlo replications to estimate the variance of the estimator under the null, as detailed in Appendix \ref{app:simSig}.}
\label{tab:1}
\end{table}

\begin{table}
\caption{\it Hypothesis testing using Theorem \ref{th:robust_clt}}
\begin{center}
\footnotesize
\begin{tabularx}{0.99\textwidth}{@{\extracolsep{\stretch{1}}}lcccccccc@{}} 
\multicolumn{9}{l}{Panel A: Size} \\
\toprule
$n$  & \multicolumn{2}{c}{$\alpha = -0.40$} & \multicolumn{2}{c}{$\alpha = -0.20$} & \multicolumn{2}{c}{$\alpha = 0$} & \multicolumn{2}{c}{$\alpha = 0.20$} \\ 
\cmidrule{2-9}
 & $p = 1$ & $p = 2$ & $p = 1$ & $p = 2$ & $p = 1$ & $p = 2$ & $p = 1$ & $p = 2$ \\ 
 $ 100 $ & $ 0.1036 $ & $ 0.1205 $ & $ 0.0920 $ & $ 0.1135 $ & $ 0.0736 $ & $ 0.1130 $ & $ 0.0973 $ & $ 0.1314 $\\
 $ 200 $ & $ 0.1064 $ & $ 0.1057 $ & $ 0.0633 $ & $ 0.0814 $ & $ 0.0561 $ & $ 0.0846 $ & $ 0.0524 $ & $ 0.0916 $\\
 $ 400 $ & $ 0.0786 $ & $ 0.0767 $ & $ 0.0553 $ & $ 0.0614 $ & $ 0.0572 $ & $ 0.0592 $ & $ 0.0419 $ & $ 0.0677 $\\
 $ 800 $ & $ 0.0639 $ & $ 0.0597 $ & $ 0.0530 $ & $ 0.0555 $ & $ 0.0523 $ & $ 0.0580 $ & $ 0.0395 $ & $ 0.0510 $\\
 $ 1600 $ & $ 0.0545 $ & $ 0.0544 $ & $ 0.0493 $ & $ 0.0507 $ & $ 0.0513 $ & $ 0.0529 $ & $ 0.0446 $ & $ 0.0506 $\\
 $ 3200 $ & $ 0.0520 $ & $ 0.0510 $ & $ 0.0506 $ & $ 0.0542 $ & $ 0.0506 $ & $ 0.0527 $ & $ 0.0438 $ & $ 0.0459 $\\
\bottomrule 
\end{tabularx}
\begin{tabularx}{0.99\textwidth}{@{\extracolsep{\stretch{1}}}lcccccccc@{}} 
\multicolumn{9}{l}{Panel B: Local power} \\
\toprule
   & \multicolumn{2}{c}{$\alpha = -0.40$} & \multicolumn{2}{c}{$\alpha = -0.20$} & \multicolumn{2}{c}{$\alpha = 0$} & \multicolumn{2}{c}{$\alpha = 0.20$} \\ 
\cmidrule{2-9}
 & $p = 1$ & $p = 2$ & $p = 1$ & $p = 2$ & $p = 1$ & $p = 2$ & $p = 1$ & $p = 2$ \\ 
 $ 100 $ & $ 0.2603 $ & $ 0.3378 $ & $ 0.2654 $ & $ 0.3605 $ & $ 0.4470 $ & $ 0.5050 $ & $ 0.8789 $ & $ 0.9561 $\\
 $ 200 $ & $ 0.3425 $ & $ 0.4011 $ & $ 0.3218 $ & $ 0.4188 $ & $ 0.3320 $ & $ 0.4725 $ & $ 0.8882 $ & $ 0.9182 $\\
 $ 400 $ & $ 0.4386 $ & $ 0.4750 $ & $ 0.4384 $ & $ 0.5111 $ & $ 0.3797 $ & $ 0.5417 $ & $ 0.7588 $ & $ 0.8041 $\\
 $ 800 $ & $ 0.5338 $ & $ 0.5894 $ & $ 0.5788 $ & $ 0.6576 $ & $ 0.5420 $ & $ 0.6543 $ & $ 0.5491 $ & $ 0.7214 $\\
 $ 1600 $ & $ 0.6758 $ & $ 0.7116 $ & $ 0.7238 $ & $ 0.7766 $ & $ 0.7274 $ & $ 0.8082 $ & $ 0.5075 $ & $ 0.7341 $\\
 $ 3200 $ & $ 0.8080 $ & $ 0.8348 $ & $ 0.8713 $ & $ 0.8972 $ & $ 0.8757 $ & $ 0.9240 $ & $ 0.5686 $ & $ 0.8078 $\\
\bottomrule 
\end{tabularx}
\end{center}
{\footnotesize \it Panel A: Empirical rejection rates of the test \eqref{eq:H00}. Panel B: Empirical rejection rates of the test \eqref{eq:H0n2}. The nominal significance level is $0.05$; the numbers in the table are the average rejection rate of $10\ 000$ Monte Carlo replications; and we used $B = 10\ 000$ Monte Carlo replications to estimate the variance of the estimator under the null, as detailed in Appendix \ref{app:simSig}.}
\label{tab:2}
\end{table}

\begin{table}
\caption{\it Hypothesis testing using Corollary \ref{cor:A}}
\begin{center}
\footnotesize
\begin{tabularx}{0.99\textwidth}{@{\extracolsep{\stretch{1}}}lcccccccc@{}} 
\multicolumn{9}{l}{Panel A: Size} \\
\toprule
$n$  & \multicolumn{2}{c}{$\alpha = -0.40$} & \multicolumn{2}{c}{$\alpha = -0.20$} & \multicolumn{2}{c}{$\alpha = 0$} & \multicolumn{2}{c}{$\alpha = 0.20$} \\ 
\cmidrule{2-9}
 & $p = 1$ & $p = 2$ & $p = 1$ & $p = 2$ & $p = 1$ & $p = 2$ & $p = 1$ & $p = 2$ \\ 
 $ 100 $ & $ 0.0580 $ & $ 0.0500 $ & $ 0.0300 $ & $ 0.0330$ & $ 0.0600 $ & $ 0.0390 $ & $ 0.0910 $ & $ 0.0580 $\\
 $ 200 $ & $ 0.0450 $ & $ 0.0380 $ & $ 0.0220 $ & $ 0.0230 $ & $ 0.0430 $ & $ 0.0260 $ & $ 0.1020 $ & $ 0.0790 $\\
 $ 400 $ & $ 0.0420 $ & $ 0.0480 $ & $ 0.0230 $ & $ 0.0210 $ & $ 0.0580 $ & $ 0.0610 $ & $ 0.0850 $ & $ 0.0690 $\\
 $ 800 $ & $ 0.0330 $ & $ 0.0330 $ & $ 0.0310 $ & $ 0.0340 $ & $ 0.0530 $ & $ 0.0490 $ & $ 0.0700 $ & $ 0.0490 $\\
 $ 1600 $ & $ 0.0390 $ & $ 0.0410 $ & $ 0.0380 $ & $ 0.0560 $ & $ 0.0550 $ & $ 0.0470 $ & $ 0.0600 $ & $ 0.0450 $\\
\bottomrule 
\end{tabularx}
\begin{tabularx}{0.99\textwidth}{@{\extracolsep{\stretch{1}}}lcccccccc@{}} 
\multicolumn{9}{l}{Panel B: Power} \\
\toprule
   & \multicolumn{2}{c}{$\alpha = -0.40$} & \multicolumn{2}{c}{$\alpha = -0.20$} & \multicolumn{2}{c}{$\alpha = 0$} & \multicolumn{2}{c}{$\alpha = 0.20$} \\ 
\cmidrule{2-9}
 & $p = 1$ & $p = 2$ & $p = 1$ & $p = 2$ & $p = 1$ & $p = 2$ & $p = 1$ & $p = 2$ \\ 
 $ 100 $ & $ 0.0750 $ & $0.0670$ & $ 0.0700 $ & $ 0.0840 $ & $ 0.2810 $ & $ 0.2780 $ & $ 0.5450 $ & $ 0.5650 $\\
 $ 200 $ & $ 0.0680 $ & $ 0.0620 $ & $ 0.1300 $ & $ 0.1390 $ & $ 0.5560 $ & $ 0.5760 $ & $ 0.7800 $ & $ 0.8050 $\\
 $ 400 $ & $ 0.0650 $ & $ 0.0700 $ & $ 0.2980 $ & $ 0.3300 $ & $ 0.8410 $ & $ 0.8660 $ & $ 0.8680 $ & $ 0.8780 $\\
 $ 800 $ & $ 0.0730 $ & $ 0.0680 $ & $ 0.6090 $ & $ 0.6250 $ & $ 0.9420 $ & $ 0.9700 $ & $ 0.8210 $ & $ 0.8380 $\\
 $ 1600 $ & $ 0.0830 $ & $ 0.1040 $ & $ 0.8830 $ & $ 0.9070 $ & $ 0.9730 $ & $ 0.9780 $ & $ 0.6370 $ & $ 0.7200 $\\
\bottomrule 
\end{tabularx}
\end{center}
{\footnotesize \it Panel A: Empirical rejection rates of the test \eqref{eq:noise_test} when $\sigma_u^2 = 0$, i.e., when $H_0$ is true. Panel B: Empirical rejection rates of the test \eqref{eq:noise_test} when $\sigma_u^2 = 0.05$, i.e., when $H_0$ is false and $H_1$ is true. The nominal significance level is $0.05$; the numbers in the table are the average rejection rate of $1\ 000$ Monte Carlo replications; and we used $B = 2\ 000$ Monte Carlo replications to estimate the variance of the estimator using $\hat{\alpha}^*$ as estimate of the fractal index used in the calculations, cf., Appendix \ref{app:simSig}.}
\label{tab:3}
\end{table}

\section{Empirical experiments}\label{sec:emp}

\subsection{Application to turbulent velocity data}\label{sec:emp_tur}
We first illustrate the use of the above methods in the context of the study of turbulent velocity flows. More precisely, we have at our disposal a time series of one-dimensional measurements of the longitudinal component of a turbulent velocity field in the atmospheric boundary layer, measured $35$ meters above ground level. The series consists of $20$ million equidistant observations over a time period of $4\ 000$ seconds, i.e. with $5\ 000$ measurements per second (in other words, a sampling frequency of $5$ kHz). The same time series has been studied in e.g. \cite{CHPP13}, \cite{BNPS13}, \cite{BHLP16}, and we refer to \cite{dhruva00} for more information on the data set itself.

We first de-mean the long time series and standardize it to have unit variance. We then sample the data at $5$ Hz, i.e. $5$ observations per second; this frequency is squarely in the so-called \emph{inertial range}, where the celebrated $5/3$-law of \cite{kolmogorov41b,kolmogorov41} states that -- in the context of this paper -- the process underlying the time series should have fractal index $\alpha = -1/6$ \citep[][Section 5]{CHPP13}.

The left plot of Figure \ref{fig:brook} contains the standardized data series sampled at $5$ Hz. The right plot gives examples of the regression \eqref{eq:reg_a} for $p\in \{1,2,3\}$. In this plot, the crosses are $\log \hat{\gamma}_p(k/n;X)$ as a function of $\log (k/n)$, while the lines are the associated OLS regression fits using a bandwidth of $m = 5$. The estimated slopes of the three lines are respectively $\hat{a}_{OLS} = 0.3361$, $0.6713$, and $1.0044$ which, through \eqref{eq:olsA}, yield the following estimates of the fractal index, $\hat{\alpha} = -0.1643$ $(0.0050)$, $-0.1638$ $(0.0047)$, and $-0.1630$ $(0.0049)$, where the numbers in parentheses denote the standard deviation of the estimates, cf. Theorem \ref{th:fbm}. Likewise, using Corollary \ref{cor:feas}, we test the the prediction of Kolmogorov via the null hypothesis
\begin{align*}
H_0: \alpha = -1/6 \quad \textnormal{against the alternative} \quad H_1: \alpha \neq -1/6,
\end{align*} 
at a nominal $5\%$ level. The corollary in this case yields P-values of $63.38\%$, $54.52\%$, and $45.54\%$, respectively. In other words, we can not reject the prediction that $\alpha = -1/6$ in this data set of turbulent velocity flows at this sampling frequency.

\begin{figure}[!t] 
\centering 
\includegraphics[width=\textwidth]{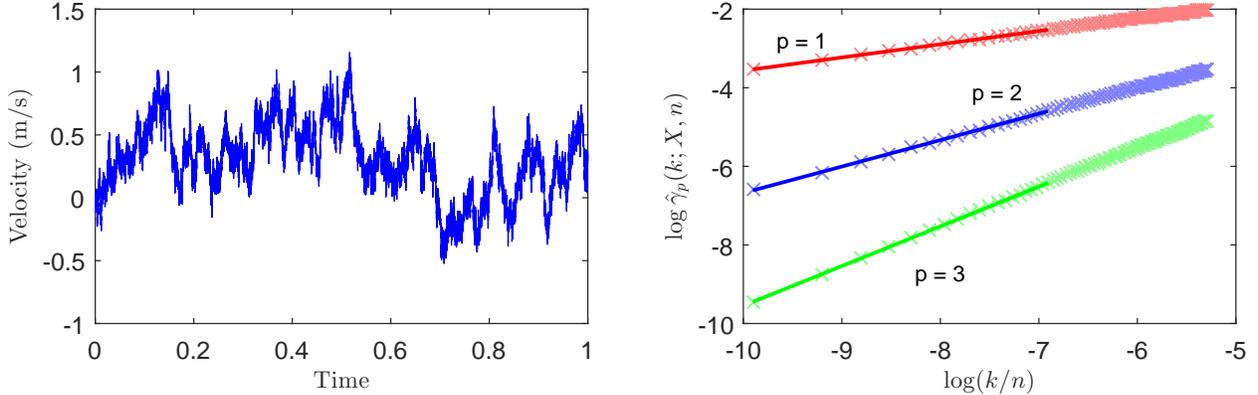}
\caption{Left: Time series plot of the standardized data sampled at $5$ Hz, as explained in the text. Right: Examples of the OLS regression \eqref{eq:reg_a} for $p = 1, 2, 3$ with bandwidth $m = 5$.}
\label{fig:brook}
\end{figure}

\subsection{Application to financial price data} \label{sec:emp_fin}

An oft-used model of high frequency financial logarithmic prices is
\begin{align}\label{eq:futprices}
Z_{j/n} = X_{j/n} + u_j, \quad j = 1, \ldots, n,
\end{align}
where $u = \{u_j\}_{j=1}^n$ is a \emph{market microstructure} noise process \citep{ohara95} and $X$ is a stochastic volatility process, e.g.,
\begin{align*}
X_t = X_0 + \int_0^t\sigma_s dG_s \quad t \geq 0.
\end{align*}

It is well known that for the absence of arbitrage in the market, it is a requirement that $X$ is a semimartingale \citep{DS94}, which under mild assumptions, is equivalent to $G$, and therefore $X$, having $\alpha = 0$. In other words, since we expect the market to be free of arbitrage, we would expect that $X$ has fractal index $\alpha = 0$. 

To test this, we study a long time series of the logarithm of financial futures prices, supposing they come from the model \eqref{eq:futprices}. To be specific, we have at our disposal data recorded every second on the front month E-mini S\&P $500$ futures contract, traded on the CME Globex electronic trading platform, from January 3, 2005 until December 31, 2014. We exclude weekends and holidays and keep only full trading days, which results in $2\ 495$ days. For these days, we further restrict attention to the most active period of the day, which is when the New York Stock Exchange (NYSE) is open, from 9.30 a.m.\ until 4 p.m.\ Eastern Standard Time (EST). This results in $23\ 400$ seconds ($6.5$ hours) for each day. We estimate $\alpha$ each day using both estimators \eqref{eq:olsA} and \eqref{eq:a_star}; this results in $N = 2\ 495$ estimates of $\alpha$ each calculated from $n = 23\ 400$ observations. 

As discussed, since we believe the market to be free of arbitrage opportunities, we would expect to find $\alpha \approx 0$ for financial price series; however, when we estimate $\alpha$ from the data using the standard OLS estimator \eqref{eq:olsA} we most often find very negative values; in fact the mean estimate of $\alpha$ over the $N = 2\ 495$ days is $-0.20$. Similarly, when we use Theorem \ref{th:fbm} to test the null hypothesis 
\begin{align}\label{eq:h0}
H_0: \alpha = 0 \quad \textnormal{against the alternative} \quad H_1: \alpha \neq 0,
\end{align} 
at a nominal $5\%$ level, we reject $H_0$ on $98.92\%$ of the days. In other words, at first glance, it seems that high frequency log prices are very rough. Of course, this finding might simply an artifact of the noise sequence $u$ which we expect to be present in high frequency stock prices; we therefore also apply the robust estimator \eqref{eq:a_star}. 

The robust estimator requires a choice of tuning parameter $\kappa$ and, as suggested in Section \ref{sec:assNoise}, we run a simulation experiment to gauge a reasonable value for this parameter. We set up our experiment to be realistic so that we can expect the optimal value of $\kappa$ we find to also be a good value to use on the real data. We therefore simulate $B = 5\ 000$ instances (``days") of an fBm with $n = 23\ 400$ observations (``seconds"); for these fBms we set $\alpha =0$, as we expect this to be the true value from the underlying process. (when $\alpha = 0$ the fBm is a standard Brownian motion.) For each simulated fBm, we add an iid sequence of Gaussian noise with $\sigma_u^2 = 0.01$; this value of $\sigma_u^2$ was chosen since it is a reasonable value for the variance of microstructure noise \citep{PeterAsgerRV06}.\footnote{The results are robust to the choice of $\sigma_u^2$ and the $\kappa$ that minimized the RMSE lies between $25$ and $200$ for all realistic values of $\sigma_u^2$.} We then apply \eqref{eq:a_star} and calculate the RMSE for various values of $\kappa \geq 2$. The results are shown in the left plot of Figure \ref{fig:ES} where we see that the RMSE is minimized for $\kappa \approx 60$.

\begin{figure}[!t] 
\centering 
\includegraphics[width=\textwidth]{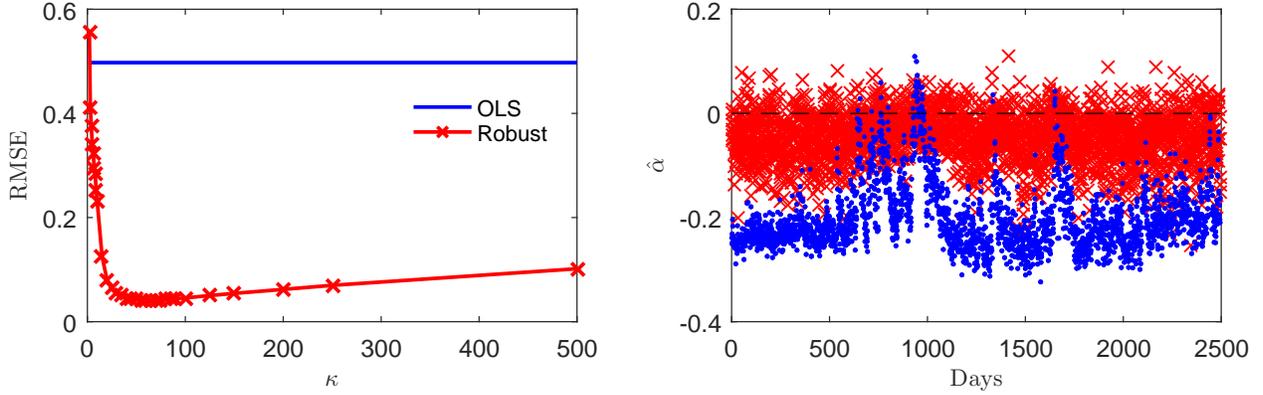}
\caption{Left: Simulation study of the root mean squared error (RMSE) of the  standard OLS estimator \eqref{eq:olsA}, blue line, and the robust estimator  \eqref{eq:a_star}, red line and crosses, as a function of $\kappa$; the details of the study are given in the text. Right: Estimates of $\alpha$ from the standard OLS estimator \eqref{eq:olsA}, blue dots, and from the robust estimator \eqref{eq:a_star} with $\kappa = 60$, red crosses. The bandwidth is $m = 5$ in both plots.}
\label{fig:ES}
\end{figure}

We then go ahead and apply both estimators to the empirical data described above; the results are seen in the right plot of Figure \ref{fig:ES}. As just described, the non-robust OLS estimator (blue dots) are often very negative (the average across all days is $-0.20$). Conversely, the robust estimator with $\kappa = 60$ (red crosses) seem much closer to the expected value of zero (the average across all days is $-0.0448$). We also apply Corollary \ref{cor:A} to the data, to test the null hypothesis, \eqref{eq:noise_test}, of no noise in the observations. When performing this test each day at a $5\%$ level, we reject the null on $99.72\%$ of the days. In other words, the formal test provides very strong evidence for the presence of market microstructure noise in the observations. 



\section{Conclusion and open problems}\label{sec:concl}
In this paper, we have laid out a large and coherent framework for analyzing data with fractal characteristics. We focused on a particular estimator of the fractal index but the methods can be applied more generally and in particular to other related estimators as well. For instance, the extension to using higher order differences -- instead of first order differences, as we used here -- when calculating the variogram, is straight forward. We showed how to extend the theory to a large class of non-Gaussian processes. The consistency result of the estimator turns out to be the same for Gaussian and non-Gaussian processes, while the central limit theorem requires an extra factor in the non-Gaussian case. The estimator of this correcting factor -- $\widehat{S_p}$ in Corollary \ref{cor:feas} -- is asymptotically irrelevant when no stochastic volatility is present. For this reason we recommend always including this factor when applying the central limit theorem of the estimator of the fractal index, whether or not one believes the underlying data to be Gaussian. Lastly, we saw that noise in the observations, e.g. measurement noise, will bias the estimates of the fractal index downwards. It is important that the practioner is cognizant of this possible bias when studying data which are potentially rough; we suggested an estimator which is robust to such noise and can be applied if one expects the data to be so contaminated.


Let us briefly comment on a few possible directions for further study. The OLS estimator relies on a bandwidth parameter $m$; in \cite{DH99} the authors find, through simulations, that the optimal value for this is $m=2$. However, as we saw above, this was not the case in general when we studied the finite sample properties of the estimator. Indeed, for rough processes, an intermediate value for the bandwidth, such as $m=5$, is preferable. How to choose $m$ optimally in finite samples is an open and interesting problem. Likewise, what value of the power parameter $p>0$ to use is an open problem; \cite{GSP12} recommend $p=1$ as they find that this makes the estimation more robust to, e.g., outliers or non-Gaussianity. However, above we proposed an estimator which is robust to such anomalies by construction, in particular when $p = 2$. An in-depth investigation of the properties of the robust estimator as compared to the standard estimator with $p=1$ would valuable. Lastly, we also saw that the robust estimator relies on the tuning parameter $\kappa$; finding a data-driven way to choose this parameter would be useful. In general, the problem of ``robust estimation" is interesting and more theoretical as well as applied research into this field would be valuable.


\appendix

\section{Proofs}\label{app:proofs}

\begin{proof}[Proof of Proposition \ref{lem:holder}]
Note first, that since $X$ is Gaussian, assumption (\ref{ass:1}) implies that for $n \geq 1$,
\begin{align*}
\E[ |X_t - X_s|^{2n}] = C_{2n} |t-s|^{(2\alpha+1)n} L(t-s)^{2n}, \quad t,s \in \R,
\end{align*}
where $x \mapsto L(x)$ is slowly varying at zero. Let now $K \subset (0,\infty)$ be a compact set and consider $t, s \in K$. By the properties of slowly varying functions \citep[][Theorem 1.5.6(ii)]{BGT}, for all $\epsilon > 0$ we can find $a>0$ such that
\begin{align*}
E[|X_t - X_s|^{2n}] \leq \tilde{C}_{1,n} |t-s|^{1 + (2\alpha +1)n-1-2n\epsilon}, \quad t-s \in (0,a],
\end{align*}
for a constant $\tilde{C}_{1,n}>0$. Conversely, since $L$ is continuous on $(0,\infty)$, we also have that
\begin{align*}
E[|X_t - X_s|^{2n}] \leq \tilde{C}_{2,n} |t-s|^{1 + (2\alpha +1)n-1}, \quad t-s>a,
\end{align*}
for a constant $\tilde{C}_{2,n}>0$. Putting these two observations together, we have that there exists a constant $\tilde{C}_{3,n}>0$ such that for all $\epsilon > 0$ we have
\begin{align*}
E[|X_t - X_s|^{2n}] \leq \tilde{C}_{3,n} |t-s|^{1 + (2\alpha +1)n-1 - 2n\epsilon}, \quad t,s \in K.
\end{align*}
Using this, we deduce that for $n$ sufficiently large, the continuity criterion of Kolmogorov shows that $X$ has a modification which is locally H\"older continuous of order $\phi$ for all $\phi \in \left(0, \frac{(2\alpha + 1)n-1-2n\epsilon}{2n} \right) = \left(0, \alpha + 1/2 - \frac{1}{2n}-\epsilon \right)$. Letting $n \uparrow \infty$, $\epsilon \downarrow 0$, yields the desired result.
\end{proof}

\begin{proof}[Proof of Proposition \ref{prop:cons}]
Note first, that we can write
\begin{align}\label{eq:aDiff}
\hat{\alpha} - \alpha = \frac{1}{p x_m^T x_m} x_m^T (U^m + \epsilon^m),
\end{align}
where
\begin{align*}
U^m &:= \left(U_{1/n}, U_{2/n}, \ldots, U_{m/n}\right)^T \\
&= \left(\log \left( \frac{\hat{\gamma}_p(1/n;X)}{\gamma_p(1/n;X)} \right), \log \left( \frac{\hat{\gamma}_p(2/n;X)}{\gamma_p(2/n;X)} \right), \ldots, \log \left( \frac{\hat{\gamma}_p(m/n;X)}{\gamma_p(m/n;X)} \right)\right)^T,
\end{align*}
and
\begin{align*}
\epsilon^m &:= \left(\epsilon_{1/n}, \epsilon_{2/n}, \ldots, \epsilon_{m/n}\right)^T = \left(\log L_p(1/n), \log L_p(2/n), \ldots, \log L_p(m/n)\right)^T.
\end{align*}

To see that the term $x_m^T \epsilon^m$ vanishes as $n \rightarrow \infty$, note that
\begin{align}\label{eq:x_vanish}
\sum_{k=1}^m x_{m,k} = \sum_{k=1}^m \left( \log k - \overline{\log m} \right) = 0
\end{align}
and therefore
\begin{align*}
x_m^T \epsilon^m = \sum_{k=1}^m x_{m,k}\log L_p(k/n) = \sum_{k=1}^m x_{m,k}\log \left( \frac{L_p(k/n)}{L_p(1/n)}  \right) \rightarrow 0, \quad n \rightarrow 0,
\end{align*}
since $\lim_{n\rightarrow \infty}  \frac{L_p(k/n)}{L_p(1/n)} = 1$ by the property of slowly varying functions.

The required result now follows by noting that 
\begin{align*}
\frac{\hat{\gamma}_p(k/n;X)}{\gamma_p(k/n;X)} = \frac{\hat{\gamma}_p(k/n;X)}{m_p\gamma_2(k/n;X)^{p/2}} \stackrel{P}{\rightarrow} 1, \qquad n \rightarrow \infty,  \quad k\geq 1,
\end{align*}
which holds by Propositon 1 in \cite{BNCP09} under (\ref{ass:lln1}); by Theorem 2 in \cite{BNCP09} under (\ref{ass:lln2}); and by Theorem 3.1. in \cite{CHPP13} under (\ref{ass:lln3}).
\end{proof}

\begin{proof}[Proof of Theorem \ref{th:fbm}]
Consider first the case where $X$ satisfies assumption (\ref{ass:clt1}). Using Theorem 2 in \cite{BNCP11} and the limit in equation \eqref{eq:rnlim}, this paper, we get
\begin{align}\label{eq:BM0}
\sqrt{n} \left( \begin{matrix} \frac{\hat{\gamma}_p(1/n;X)}{ \gamma_p(1/n;X)} - 1 \\
					 \vdots \\
					 \frac{\hat{\gamma}_p(m/n;X)}{\gamma_p(m/n;X)} -1	
		\end{matrix} \right) \stackrel{d}{\rightarrow} N(0,\Lambda_p), \quad n \rightarrow \infty,
\end{align}
where $\Lambda_p = \{ \lambda_p^{k,v} \}_{k,v = 1}^m$ is a $m \times m$ matrix with entries
\begin{align}\label{eq:lamLim}
\lambda_p^{k,v} =   \lim_{n\rightarrow \infty}  n \cdot  Cov \left(\frac{\hat{\gamma}_p(k/n;B^H)}{ \gamma_p(k/n;B^H)},\frac{\hat{\gamma}_p(v/n;B^H)}{ \gamma_p(v/n;B^H)}  \right), \quad k,v = 1, 2, \ldots, m,
\end{align}
with $\gamma_p(\cdot;B^H)$ denoting the $p$'th order variogram for a fractional Brownian motion with Hurst index $H = \alpha + 1/2$, and similarly for $\hat{\gamma}_p$. Note that the limit in \eqref{eq:lamLim} exists for $k,v = 1, 2, \ldots m$, by \cite{BM83}, Theorem 1, see also \cite{CHPP13}, Remark 3.3. 
 
As we will show below, assumption (\ref{ass:99}) implies that
\begin{align}\label{eq:epsConv}
\sqrt{n} x_m^T \epsilon^m \rightarrow 0, \quad n \rightarrow \infty,
\end{align}
which means that from \eqref{eq:aDiff} we get, using \eqref{eq:BM0} and the delta method,
\begin{align*}
\sqrt{n}\left(\hat{\alpha} - \alpha \right) \stackrel{d}{\rightarrow} N\left(0,  \frac{x_m^T \Lambda_p x_m}{(x_m^T x_m)^2 p^2}   \right), \quad n \rightarrow \infty,
\end{align*}
which is what we wanted to show. To see that \eqref{eq:epsConv} holds, use the rule of l'H\^opital, and the properties of slowly varying functions, to conclude
\begin{align*}
\lim_{n\rightarrow \infty} \sqrt{n} \log \left( \frac{L_p(k/n)}{L_p(1/n)} \right) = \lim_{n\rightarrow \infty} 2n^{-1/2}\left( \frac{L_p'(k/n)}{L_p(k/n)}k - \frac{L_p'(1/n)}{L_p(1/n)}\right) = 0,
\end{align*}
by assumption (\ref{ass:99}). This concludes the proof when $X$ is Gaussian.

Suppose instead that assumption (\ref{ass:clt3}) holds (the case (\ref{ass:clt2}) is similar). We proceed analogously to above: by Theorem 4 of \cite{BNCP11}, see also Theorem 3.2. and Remark 3.4. of \cite{CHPP13}, we get
\begin{align*}
\sqrt{n} \left( \begin{matrix} \frac{\hat{\gamma}_p(1/n;X)}{ \gamma_p(1/n;X)} - 1 \\
					 \vdots \\
					 \frac{\hat{\gamma}_p(m/n;X)}{\gamma_p(m/n;X)} -1	
		\end{matrix} \right) \stackrel{st}{\rightarrow} \int_0^1 \sigma_s^p \Lambda_p dBs,
\end{align*}
where $B$ is an $m$-dimensional Brownian motion, defined on an extension of the original probability space, $(\Omega, \mathcal{F}, \mathbb{P})$, independent of $\mathcal{F}$. The matrix $\Lambda_p$ is identical to the one above.

We proceed as before. In particular, invoking the delta method we get
\begin{align*}
\sqrt{n}\left(\hat{\alpha} - \alpha \right) \stackrel{st}{\rightarrow} \frac{x_m^T \Lambda_p}{x_m^Tx_m p} \frac{\int_0^1 \sigma_s^p dB_s}{\int_0^t \sigma_s^pds},
\end{align*}
or, in other words (conditionally on $(\sigma_t)_{t\in \R}$),
\begin{align*}
\sqrt{n}\left(\hat{\alpha} - \alpha \right) \stackrel{st}{\rightarrow} Z_p \cdot S_p, \quad S_p := \frac{\sqrt{\int_0^1 \sigma_s^{2p} ds}}{\int_0^1 \sigma_s^pds},
\end{align*}
and where $Z_p$ is as in Theorem \ref{th:fbm}. This concludes the proof.
\end{proof}

\begin{proof}[Proof of Proposition \ref{prop:S2fbm}]
(i) Note that we can write
\begin{align*}
\widehat{S_p} = \frac{\sqrt{m_{2p}^{-1} \hat{\gamma}_{p}(1/n;X)/\gamma_2(1/n;X)^{p}}}{m_p^{-1} \hat{\gamma}_p(1/n;X)/\gamma_2(1/n;X)^{p/2}}.
\end{align*}
The result now follows from Proposition 1 in \cite{BNCP09}.

(ii) This part follows from Theorem 3.1. in \cite{CHPP13}.
\end{proof}

\begin{proof}[Proof of Proposition \ref{prop:incons}]
This follows easily from the fact that when $\sigma_u^2>0$ it holds for any $k\geq 1$
\begin{align*}
\hat{\gamma}_p(k/n;Z) \stackrel{P}{\rightarrow} \gamma_p(0;Z) = C_p \sigma_u^p > 0, \quad n \rightarrow \infty,
\end{align*}
so that
\begin{align*}
\hat{a}_{OLS} =   \frac{1}{x_m^T x_m} x_m^T \hat{\gamma}_p^m  \stackrel{P}{\rightarrow}  0, \quad n \rightarrow \infty,
\end{align*}
by \eqref{eq:x_vanish}.
\end{proof}

\subsection{Proofs related to the noise robust estimator}
Most of the proofs related to the noise robust estimator $\hat{\alpha}^*$ proceeds analogously to the proofs related to the standard estimator $\hat{\alpha}$ given above. Indeed, the function $f_p(h;Z,\kappa)$ has the same behavior as $\gamma_2(h;X)$ for $h \approx 0$ --- cf. equations \eqref{eq:gam_rel} and \eqref{eq:f_rel}. The upshot is that most of the theory will go through as in the proofs above. The following Lemma formally provides the details.

\begin{lemma}\label{lem:L_star}
Suppose the process $X$ is such that its variogram $\gamma$ satisfies assumptions (\ref{ass:1})--(\ref{ass:99}). Fix $\kappa > 1$ and define the function $f$ as in \eqref{eq:f_rel} with
\begin{align*}
L^*(x;\kappa) := \kappa^{2\alpha+1}L(x\kappa) - L(x), \quad x>0.
\end{align*}
Then $f$ satisfies assumptions (\ref{ass:1})--(\ref{ass:99}) with $L^*$ as the slowly varying function.
\end{lemma}

\begin{proof}[Proof of Lemma \ref{lem:L_star}]
That (\ref{ass:1}) and (\ref{ass:99}) are satisfied is trivial. To see that (\ref{ass:2}) holds, note that
\begin{align*}
f(x;X,\kappa) = x^{2\alpha+1}L^*(x;\kappa) = \gamma_2(x\kappa) - \gamma(x), \quad x>0,
\end{align*}
so that
\begin{align*}
\frac{\partial^2}{\partial x^2}f(x;X,\kappa) = x^{2\alpha-1} L_{(2)}^*(x;\kappa),
\end{align*}
where
\begin{align*}
L_{(2)}^*(x;\kappa) := \kappa^{2\alpha+1} L_2(x\kappa) - L_2(x)
\end{align*}
is slowly varying at zero and continuous on $(0,\infty)$, since these properties hold for $L_2$ by assumption. 

With this function $L_{(2)}^*$, it is also the case that (\ref{ass:3}) holds. To see this, write
\begin{align*}
\left| \frac{L_{(2)}^*(y;\kappa)}{L^*(x;\kappa)}\right| &\leq \left| \frac{ \kappa^{2\alpha+1}\frac{L_{2}(y\kappa)}{L(\kappa x)}}{\kappa^{2\alpha+1} - L(x)/L(\kappa x)}\right| +   \left| \frac{ \frac{L_{2}(y)}{L(x)}}{\kappa^{2\alpha+1}L(\kappa x)/L(x) - 1}\right| \\
	&=  \left( \left| \frac{ \kappa^{2\alpha+1}\frac{L_{2}(y\kappa)}{L_2(y)}\frac{L(x)}{L(\kappa x)}}{\kappa^{2\alpha+1} - L(x)/L(\kappa x)}\right| +   \left| \frac{1}{\kappa^{2\alpha+1}L(\kappa x)/L(x) - 1}\right| \right)\frac{L_{2}(y)}{L(x)}.
\end{align*}
Using that both $L$ and $L_2$ are slowly varying at zero and obey (\ref{ass:3}) by assumption, the result follows.
\end{proof}

\begin{proof}[Proof of Proposition \ref{prop:robust_cons}]
Given Lemma \ref{lem:L_star}, the proof is similar to the one of Proposition \ref{prop:cons}; we skip the details.
\end{proof}

\begin{proof}[Proof of Theorem \ref{th:robust_clt}]
\ref{it:i} Given Lemma \ref{lem:L_star}, the proof is similar to the one of Theorem \ref{th:fbm}; we skip the details. The asymptotic variance can be calculated using the delta method -- we give the expression in Appendix \ref{app:avar}. 

\ref{it:ii} For ease of notation, we prove this for $p=2$. The case with general $p>0$ follows by (conditional) Gaussianity. Write
\begin{align*}
\hat{f}_2(k/n;Z,\kappa) = \hat{f}_2(k/n;X,\kappa) + \hat{f}_2(k/n;u,\kappa) + 2\hat{f}_{1,1}(k/n;X,u,\kappa),
\end{align*}
where
\begin{align*}
\hat{f}_{1,1}(k/n;X,u,\kappa) := \hat{\gamma}_{1,1}(\kappa k/n;X,u)-\hat{\gamma}_{1,1}(k/n;X,u),
\end{align*}
with
\begin{align*}
\hat{\gamma}_{1,1}(k/n;X,u) := \frac{1}{n- k}\sum_{i=1}^{n- k} \big( X_{\frac{i+ k}{n}} - X_{\frac{i}{n}}  \big) \big( u_{i+ k} - u_{i} \big).
\end{align*}
We are interested in the asymptotic behavior of 
\begin{align*}
\frac{\hat{f}_2(k/n;Z,\kappa)}{f_2(k/n;X;\kappa)} = \frac{\hat{f}_2(k/n;X,\kappa)}{f_2(k/n;X;\kappa)} + \frac{\hat{f}_2(k/n;u,\kappa)}{f_2(k/n;X;\kappa)} +\frac{\hat{f}_{1,1}(k/n;X,u,\kappa)}{f_2(k/n;X;\kappa)}.
\end{align*}
From \ref{it:i}, we know that the first of these terms obeys a CLT with rate $\sqrt{n}$. Further, it is standard to show that $\hat{f}_2(k/n;u,\kappa)$ does the same --- hence, the second term will diverge, since the denominator $f_2(k/n;X;\kappa) \rightarrow 0$ as $n\rightarrow 0$. The result follows.
\end{proof}

\begin{proof}[Proof of Theorem \ref{th:A}]
\ref{it:i} Note first, that we can write
\begin{align}\label{eq:AmA}
\hat{\alpha}^* - \hat{\alpha} = \frac{1}{2x_m^T x_m}x_m^T \log \hat{f}_p^m - \frac{1}{p x_m^T x_m}x_m^T \log \hat{\gamma}_p^m = \frac{1}{2x_m^T x_m}x_m^T \log \left( \frac{\hat{f}_p^m}{(\hat{\gamma}_p)^{2/p}}  \right),
\end{align}
where $\hat{\gamma}_p^m$ and $\hat{f}_p^m$ denotes $m\times 1$ vectors as explained in the main text. We see that \eqref{eq:AmA} takes the form as the OLS estimators studies above. Given Lemma \ref{lem:L_star}, the proof of the present theorem is similar to the one of Theorem \ref{th:fbm}; we skip the details. The asymptotic variance can be calculated using the delta method -- we give the expression in Appendix \ref{app:avar}. 

\ref{it:ii} This is obvious due to the fact that
\begin{align*}
 \hat{\alpha}^* - \hat{\alpha} \stackrel{P}{\rightarrow} \alpha + 1/2 > 0, \quad \textnormal{as } n\rightarrow \infty,
\end{align*}
due to Propositions  \ref{prop:incons} and \ref{prop:robust_cons}.
\end{proof}

\section{Monte Carlo approximation of the variance of the OLS estimator}\label{app:simSig}
For a number of observations $n\in \N$, we seek to calculate the finite sample versions of the entries of the matrix $\Lambda_p = \{ \lambda_p^{k,v} \}_{k,v=1}^m$ given by \eqref{eq:lamLimTh}, i.e.
\begin{align*}
\lambda_{p,n}^{k,v} &:=    n \cdot Cov \left(\frac{\hat{\gamma}_p(k/n;B^H)}{ \gamma_p(k/n;B^H)},\frac{\hat{\gamma}_p(v/n;B^H)}{ \gamma_p(v/n;B^H)}  \right) \\
	&= \frac{n}{ \gamma_p(k/n;B^H) \gamma_p(v/n;B^H)} \cdot  Cov \left(\hat{\gamma}_p(k/n;B^H),\hat{\gamma}_p(v/n;B^H) \right) , \quad k,v = 1, 2, \ldots, m.
\end{align*}
First note that for $p>0$,
\begin{align*}
\gamma_p(k/n;B^H) = C_p (k/n)^{pH}, \quad C_p = \frac{2^{p/2}}{\sqrt{\pi}} \Gamma\left(\frac{p+1}{2}\right),
\end{align*}
so what remains is to calculate the covariance term. We suggest to approximate this term by Monte Carlo simulation as follows. First, pick a large number $B \in \N$ of Monte Carlo replications. Then, for each $b = 1, \ldots, B$:
\begin{enumerate}
\item
	Simulate $n$ observations of an fBm on $[0,1]$ with Hurst index $H=\alpha + 1/2$.
\item
	Calculate the values of the empirical variograms $\hat{\gamma}^{(b)}_p(k/n;B^H)$ for $k = 1, \ldots, m$ using equation \eqref{eq:gam_hat}.
\end{enumerate}
With these $B$ instances of the empirical variograms, estimate the relevant covariances. The asymptotic variance can be approximated by choosing $n$ very large.

\section{Expressions for the asymptotic variance in theorems \ref{th:robust_clt} and \ref{th:A} }\label{app:avar}
Let $k^*$ be the smallest integer such that $k^* \cdot \kappa > m$. As can be seen in the proofs above, we are interested in the joint asymptotic distribution of functions of the random quantities $\frac{\hat{\gamma}_p(k/n;X)}{\gamma_p(k/n;X)}$ for $k\in \mathcal{M}$, where
\begin{align*}
\mathcal{M} := \{1, 2, \ldots, m, k^*\kappa, (k^*+1)\kappa, \ldots, m\kappa\}.
\end{align*}
Let $|A|$ denote the cardinality of the set $A$ and define the $|\mathcal{M}| \times |\mathcal{M}|$ matrix $\Lambda_p^* = \{ \lambda_p^{k,v,*} \}_{k,v\in \mathcal{M}}$ as
\begin{align}\label{eq:LAM_star}
\lambda_p^{k,v,*} =   \lim_{n\rightarrow \infty} n \cdot Cov \left(\frac{\hat{\gamma}_p(k/n;B^H)}{ \gamma_p(k/n;B^H)},\frac{\hat{\gamma}_p(v/n;B^H)}{ \gamma_p(v/n;B^H)}  \right), \quad k,v \in \mathcal{M},
\end{align}
where $B^H$ is an fBm with Hurst index $H = \alpha + 1/2$. The values of the entries in this matrix can be approximated by Monte Carlo simulation in the same way as described in Appendix \ref{app:simSig}.

\subsection{Asymptotic variance of Theorem \ref{th:robust_clt}}
By the delta method, we have that
\begin{align*}
\sigma^{2,*}_{m,p}= \frac{x_m^T \Sigma_1^T \Lambda^*_p \Sigma_1 x_m}{4 (x_m^T x_m)^2},
\end{align*}
where $\Lambda_p^*$ is defined in equation \eqref{eq:LAM_star} and $\Sigma_1$ is the $m \times  |\mathcal{M}|$ matrix with entries
\begin{align*}
\Sigma_1(i,j) = \frac{2}{p\left(\kappa^{2\alpha+1}-1\right)} \cdot
	\begin{cases}
    		-1       & \quad \text{for } (i,j) = (1,1), (2,2), \ldots (m,m), \\
    		\kappa^{2\alpha+1}  & \quad \text{for } (i,j) = (1,\kappa), (2,2\kappa), \ldots (k^*-1,(k^*-1)\kappa), \\
		\kappa^{2\alpha+1}  & \quad \text{for } (i,j) = (k^*,m+1), (k^*+1,m+2), \ldots (m,|\mathcal{M}|), \\
		0  & \quad \text{else.}
  	\end{cases}
\end{align*}


\subsection{Asymptotic variance of Theorem \ref{th:A}}
By the delta method, we have that
\begin{align*}
\sigma^{2,**}_{m,p}= \frac{x_m^T \Sigma_2^T \Lambda^*_p \Sigma_2 x_m}{4 (x_m^T x_m)^2},
\end{align*}
where $\Lambda_p^*$ is defined in equation \eqref{eq:LAM_star} and $\Sigma_2$ is the $m \times  |\mathcal{M}|$ matrix with entries
\begin{align*}
\Sigma_2(i,j) = \frac{2\kappa^{2\alpha+1}}{p\left(\kappa^{2\alpha+1}-1\right)} \cdot
	\begin{cases}
    		-1       & \quad \text{for } (i,j) = (1,1), (2,2), \ldots (m,m), \\
    		1  & \quad \text{for } (i,j) = (1,\kappa), (2,2\kappa), \ldots (k^*-1,(k^*-1)\kappa), \\
		1  & \quad \text{for } (i,j) = (k^*,m+1), (k^*+1,m+2), \ldots (m,|\mathcal{M}|), \\
		0  & \quad \text{else.}
  	\end{cases}
\end{align*}

{\small 
\bibliographystyle{chicago}
\bibliography{mybib-v6}
}

\end{document}